\documentclass[preprint,11pt]{elsarticle}
\usepackage[T1]{fontenc}
\usepackage{geometry}
\usepackage{mwe}
\usepackage{xcolor}
\usepackage{longtable}
\usepackage{ltablex}
\usepackage{caption, booktabs}
\usepackage{adjustbox}
\usepackage{soul}
\usepackage{mathrsfs}
\usepackage{array}
\usepackage{color,soul}
\usepackage{colortbl}
\usepackage[normalem]{ulem}
\usepackage{multicol}
\usepackage[colorlinks]{hyperref}
\usepackage{doi}
\usepackage{marginnote}
\usepackage{tikz}
\usepackage{array}
\reversemarginpar 
\usepackage{amsmath,amssymb,amsthm,marvosym,tikz,mathdots,caption,yhmath,MnSymbol} %AMS symbols vs Mn Symbol check??
\usepackage{tikz}
\usepackage{pdflscape}
\usepackage{tikz-cd}
\usetikzlibrary{calc,arrows}
\usetikzlibrary{shapes.geometric}
\usepackage{wrapfig}
\usepackage{subcaption}
\usepackage{pdflscape}
\usepackage{blindtext}
\usepackage{multirow}
\usepackage{graphicx}
\newcommand\numberthis{\addtocounter{equation}{1}\tag{\theequation}}
\usetikzlibrary{backgrounds}
\usetikzlibrary{patterns,shapes.misc, positioning}
\usepackage{pgfplots}
\usepackage[overload]{empheq}

\newcommand{\SMR}{\textnormal{SMR}}
\newcommand{\MBPD}{\textnormal{MBPD}}
 
 \newcommand{\AG}{\operatorname{AG}}

\usetikzlibrary{snakes}

\newtheorem{theorem}{Theorem}[section]
\newtheorem{lemma}[theorem]{Lemma}
\newtheorem{corollary}[theorem]{Corollary}

\newtheorem{definition}[theorem]{Definition}
\newtheorem{example}[theorem]{Example}

\newtheorem{proposition}[theorem]{Proposition}

\newtheorem{note}[theorem]{Note}
 \newtheorem{illustration}[theorem]{Illustration} 

\allowdisplaybreaks
 \usepackage{comment}
 \usepackage{theoremref}
 \usepackage{hyperref}

 \usepackage{bibentry}
\begin{document}\begin{frontmatter}		
\title{On modular balanced partition designs}
\author[a]{S. Karthik}
\ead{karthik_p210136ma@nitc.ac.in}	
\author[b]{Peter J. Cameron}
\ead{pjc20@st-andrews.ac.uk}
\author[a]{Krishnan Paramasivam}
\ead{sivam@nitc.ac.in}
\address[a]{Department of Mathematics, National Institute of Technology Calicut \\ Kozhikode~\textnormal{673601}, India. }
  		\address[b]{School of Mathematics and Statistics, University of St Andrews, St Andrews, Fife KY16
9SS, UK}

\begin{abstract}
Let $X$ be a finite set of integers with cardinality $\nu = \kappa \lambda$. A \emph{modular balanced partition design} is a triplet $(X, \mathcal{A}, \mathcal{B})$ satisfying the following conditions:
\begin{enumerate}
    \item[(i)] $\mathcal{A}$ is a partition of $X$ into $\kappa$ blocks of size $\lambda$, such that every element of $X$ appears in exactly one block. If $\mathcal{A} = \{A_1, A_2, \ldots, A_{\kappa}\}$, then
    $
        \sum_{a\in A_i} a \equiv i \lambda \pmod{\nu}$, for each $i=1,2,\ldots,\kappa.
    $
    \item[(ii)] $\mathcal{B}$ is a partition of $X$ into $\lambda$ blocks of size $\kappa$, such that every element of $X$ appears in exactly one block. If $\mathcal{B} = \{B_1, B_2, \ldots, B_{\lambda}\}$, then
    $\sum_{b\in B_j} b \equiv j \kappa \pmod{\nu}$, for each $j=1,2,\ldots,\lambda.
    $
    \item[(iii)] $A_i \cap B_j$ has exactly one element, for any $A_i \in \mathcal{A}$, and $B_j \in \mathcal{B}$.
\end{enumerate}

In this article, we prove the necessary conditions for the existence of a modular balanced partition design. Moreover, we investigate and identify a relationship between a modular balanced partition design and a subgroup magic rectangle. Then by using affine automorphisms of an Abelian group, we prove the existence of non-isomorphic modular balanced partition designs. Finally, we provide a method to construct a transversal design via a modular balanced partition design.
 
\noindent \textbf{Keywords:} Modular congruence, partition, transversal design, block design, affine automorphism, subgroup magic rectangle.\\
  {2020 Mathematical Subject Classification:   05B05, 05B30, 05B99, 20K27}
\end{abstract}
%--------------------------
\end{frontmatter}

%--------------------------

\section{Introduction}

The study of additive properties within sets of integers and finite Abelian groups has been a central topic in combinatorial number theory. A classical problem in this area is the existence of {zero-sum subsets and partitions}, where elements of a set or group are divided into subsets whose elements sum to zero under addition or the group operation. This line of inquiry originated from the seminal result of Erdős, Ginzburg, and Ziv~\cite{Erdos1961}, which established that any sequence of $2n-1$ integers contains a subsequence of length $n$ whose sum is zero modulo $n$. This theorem and its extensions form the cornerstone of zero-sum theory. Over the years, the {zero-sum partition problem} has been extensively generalised and studied across various group-theoretic frameworks. In 2006, Gao and Geroldinger~\cite{MR2313123} provided a comprehensive survey of these developments, emphasizing the structural and extremal aspects of the zero-sum problems in finite Abelian groups. Further, algebraic treatments were developed by Geroldinger and Halter-Koch~\cite{zbMATH05018184}, who linked zero-sum phenomena with factorization theory, while Grynkiewicz~\cite{zbMATH06162097} discussed deeper structural results and invariants associated with additive partitions.

A natural extension of this framework is the concept of a {constant-sum partition}, where each subset in the partition has the same non-zero sum rather than a zero sum. Recent works, particularly those by Cichacz~\cite{MR3720542,MR3761934}, explored such constant-sum and zero-sum partitions within finite Abelian groups and investigated their applications in graph labeling problems and combinatorial designs. These studies revealed that the zero-sum and the constant-sum partitions are closely related, representing two complementary aspects of additive regularity in groups. Moreover, constant-sum partitions have found relevance in constructing balanced combinatorial structures such as group distance-magic labeling problems, difference families, and modular array designs, which connect algebraic group properties with combinatorial uniformity.

However, an intriguing question arises, when the condition of uniformity is relaxed: \emph{What happens if the subset sums are all distinct constants rather than equal or zero.$?$} This situation gives rise to a new combinatorial setting, where partitions are characterized not by uniform additive behavior; but by structured diversity in their subset sums. Investigation of such {distinct-sum} or {variable-constant partitions} opens new directions in additive combinatorics and provides potential applications in array constructions, group-based designs, and cryptographic structures.

\section{Preliminaries}
Throughout this article, $\Gamma$ is a finite additive Abelian group with identity element $0$. If $H$ is a subgroup of $\Gamma$, denoted by $H\le \Gamma$, then the left coset of \( H \) is a set
\(
x+H= \{x+h : h \in H\}\), where \( x \in \Gamma \) is some fixed element. Moreover, the union of all the left cosets of \( H\) equals \( \Gamma \), all the left cosets of \( H\) are of the same size as \( H \), and the set of all the left cosets of \( H \) forms a partition of \( \Gamma \).  Further, $[\Gamma : H]$ is the index of the subgroup $H$ in $\Gamma$. A non-zero element $g\in \Gamma$ is an involution if $g=-g$, where $-g$ is the inverse of $g$ in $\Gamma$ and, in fact, $g \in \Gamma$ is involution if and only if $o(g)=2$.  We use $I(\Gamma)$ to denote the set of all involutions in $\Gamma$.

\begin{lemma}[\cite{MR2239315}]{\label{lemma1.4}}
 Let $\displaystyle\sum_{g \in \Gamma} g = h$. If $I(\Gamma)=\{g_\iota\}$, then  $h=g_\iota$, otherwise $h=0$.
\end{lemma}

Let $\Gamma$ be a group acting on a non-empty set $X$. For an element $x \in X$, the \emph{orbit} of $x$ under the action of $\Gamma$ is the set
\(
\operatorname{Orb}_\Gamma(x) = \{\, g + x : g \in \Gamma \,\}
\) and the \emph{stabiliser} of $x$ in $\Gamma$ is defined as
\(
\operatorname{Stab}_\Gamma(x) = \{\, g \in \Gamma : g + x = x \,\},
\)
which forms a subgroup of $\Gamma$ consisting of all elements that fix $x$ under the group action. Further, a group action is equivalent to a group homomorphism $\psi: \Gamma \rightarrow  Sym(X)$, where $Sym(X)$ is the symmetric group on $X$.

\begin{theorem}[Orbit-Stabilizer Theorem \cite{zbMATH05233000}]{\label{ost}}
    Let $\Gamma$ be a group acting on $X$ and $x \in X$. Then 
\begin{itemize}
    \item[$a)$]  the stabiliser \(
\operatorname{Stab}_\Gamma(x)\) is a subgroup \( \Gamma\), and
  \item[$b)$] there is a bijection between the orbit of $x$ and the set of right cosets of \(
\operatorname{Stab}_\Gamma(x)\) 
in $\Gamma$.
\end{itemize}

\end{theorem}
\begin{lemma}[Burnside's Lemma \cite{zbMATH05233000}]
\label{lem:burnside}
Let a finite group $\Gamma$ act on a finite set $X$. For each $g\in \Gamma$, define
\(
\operatorname{Fix}(g)=\{x\in X:g+x=x\}.
\)
Then the number of orbits of $X$ under the action of $\Gamma$ is
\[
|X/\Gamma|
=
\frac{1}{|\Gamma|}
\sum_{g\in \Gamma}
|\operatorname{Fix}(g)|.
\]
\end{lemma}
 
For more group theory  notations and terminology, we refer to Cameron \cite{zbMATH05233000}, Warner\cite{MR1068318} and Silverman \cite{MR4423367}. In this article, for two integers $i$ and $j$ with $i<j$, we use $[i,j]$ to denote the set of integers $i,i+1,\ldots, j$. Now, 
we recall some important results from partitions of Abelian group.  In 2009, Kaplan et al.~\cite{MR2510327} introduced the concept of zero-sum partitions for subsets in the Abelian group $\Gamma$.
Let $\Gamma$ be an Abelian group and let $A$ be a finite subset of $\Gamma\setminus\{0\}$ with $|A|=n$.  
Then the set $A$ is said to have the \emph{zero-sum partition property} if for every partition
\(
n = r_1 + r_2 + \cdots + r_t
\)
of $n$ with $r_i \ge 2$ for $i \in [1,t]$, there exists a partition of $A$ into pairwise disjoint subsets
$A_1,A_2,\dots,A_t$ such that $|A_i|=r_i$ and
\(
  \sum_{a\in A_i} a = 0, ~  i \in [1,t].
\) Later, Cichacz~\cite{MR3761934} considered a  restricted version of this property, called the \emph{$m$-zero-sum partition property}. 
In this case, if $m$ divides $n$ with $m\ge 2$, the set (or group) can be partitioned into pairwise disjoint subsets $A_i$ such that $|A_i|=m$ and
\(
  \sum_{a\in A_i} a = 0 ,  i \in [1,t].
\)
Thus, the $m$-zero-sum partition property can be viewed as a specialization of the zero-sum partition property in which all parts of the partition have the same size. In 2019, Freyberg~\cite{MR4145425} further generalized this idea by introducing the notion of a \emph{constant sum partition} of $\Gamma$, where
\(
 \sum_{a \in A_i} a = g, ~g \in \Gamma,
\) $i \in [1,t]$. 
For more details, see~\cite{MR3720542,MR2313123,MR3418508}.

% In 2009, Kaplan et al. \cite{MR2510327} introduced the concept of zero-sum partitions for subsets in the Abelian group $\Gamma$.
% Let $\Gamma$ be an Abelian group and let $A$ be a finite subset of $\Gamma\setminus\{0\}$ with $|A|=n$.  
% Then the set $A$ is said to have the  {zero-sum-partition property} if for every partition
% \(
% n = r_1 + r_2 + \cdots + r_t
% \)
% of $n$ with $r_i \ge 2$ for $1\le i\le t$, there exists a partition of $A$ into pairwise disjoint subsets
% $A_1,A_2,\dots,A_t$ such that $|A_i|=r_i$ and
% \(\displaystyle
% \sum_{a\in A_i} a = 0 \quad\text{for } 1\le i\le t.
% \)  In 2018, Cichacz \cite{MR3761934}, extended the zero sum partition property to the $m$-zero sum partition property, if $m$ divides $n, m \geqslant 2$ and there is a partition of $\Gamma$ into pairwise disjoint subsets \(A_i\) such that $\left|A_i\right|=m$ and $\sum_{a \in A_i} a=0$ for $1 \le i \le t $. In 2019, Freyberg \cite{MR4145425}, generalized it as a {constant sum partition} of $\Gamma$ by $\displaystyle\sum_{a \in A_i} a = g, g \in \Gamma$, for every $1 \le i \le t  $.
% For more details, see \cite{MR3720542,MR2313123,MR3418508}.

As a generalization of constant-sum partitions of an Abelian group and group magic rectangle \cite{MR4149158}, the authors introduced the concept of subgroup magic rectangle \cite{karthik2025existence}, in which the sets of all individual row sums and all individual column sums, are distinct.

\begin{definition}[\cite{karthik2025existence}]Let $\Gamma$ be a finite additive Abelian group. A subgroup magic rectangle is an $m \times n$ array $\mathbb{S}=(\zeta_{ij})$, where $\zeta_{ij} \in \Gamma$ for all $i,j$ and $|\Gamma|=mn$, each element appears exactly once in such a way that the set of all row-sums forms a subgroup of order $m$ and the set of all column-sums forms a subgroup of order $n$. It's denoted as $\SMR_{\Gamma}(m,n : H_1, H_2)$, where $H_1, H_2 \le \Gamma$.
\end{definition}

In addition, we recall necessary fundamentals and important classes of combinatorial designs that will be useful in the subsequent discussion. 
Combinatorial design theory deals with the arrangement of elements of a finite set into subsets, called \emph{blocks}, in such a way that certain balance or uniformity properties are satisfied. 
These structures play a central role in combinatorics and have numerous connections with finite geometry, coding theory, and group-based constructions.

\begin{definition}[\cite{zbMATH00050655}]
A $t$-design with parameters $(v,k,\lambda)$ (or a $t$-$(v,k,\lambda)$ design) is a pair 
$\mathcal{D}=(X,\mathcal{B})$, where $X$ is a set of `points' of cardinality $v$, and 
$\mathcal{B}$ is a collection of $k$-element subsets of $X$ called `blocks', with the 
property that any $t$ points are contained in precisely $\lambda$ blocks.\end{definition}

To rule out trivial or degenerate cases, we assume that both $X$ and 
$\mathcal{B}$ are non-empty and that the parameters satisfy $v \ge k \ge t$, 
which ensures that $\lambda > 0$. A $t$-design with $\lambda = 1$ is  
a \emph{Steiner system}. Such system is commonly denoted by $S(t,k,v)$, and 
the notation is sometimes generalized to $S_{\lambda}(t,k,v)$ when referring 
to designs with an arbitrary value of $\lambda$. A $2$-design is commonly referred to as a \emph{balanced incomplete-block design} (BIBD) or shortly, a \emph{block design}. 
The term ``balanced'' reflects the condition that every pair of points 
occurs together in the same number of blocks, while ``incomplete'' indicates 
that the size of each block is strictly smaller than the number of points. Formally, a balanced incomplete-block design is a pair 
$(X,\mathcal{B})$, where $X$ is a set of $v$ points and 
$\mathcal{B}$ is a collection of subsets of $X$ called blocks such 
that $|B|=k$ for every $B\in\mathcal{B}$, and every pair of distinct points 
of $X$ appears together in exactly $\lambda$ blocks. 
Such a design is denoted by $\mathrm{BIBD}(v,k,\lambda)$. 
Suppose $(X,\mathcal{A})$ and $(\mathcal{Y},\mathcal{B})$ are two block designs with $|X|=|\mathcal{Y}|$. We say that $(X,\mathcal{A})$ and $(\mathcal{Y},\mathcal{B})$ are \emph{isomorphic} if there exists a bijection $\alpha:X\to\mathcal{Y}$ such that a subset $A$ of $X$ is a block of $\mathcal{A}$ if and only if the image $\alpha(A)=\{\alpha(x):x\in A\}$ is a block of $\mathcal{B}$. In this case, the bijection $\alpha$ is called an \emph{isomorphism} between the block designs. When $(X,\mathcal{A})=(\mathcal{Y},\mathcal{B})$, an isomorphism $\alpha:(X,\mathcal{A})\to(X,\mathcal{A})$ is called an \emph{automorphism} of the block design $(X,\mathcal{A})$.

Another important combinatorial design closely related to Latin square and orthogonal array, is the {transversal design}\cite{MR1370815}, which acts as a bridge between combinatorial design and additive group theory. A transversal design $\mathrm{TD}_{\lambda}(k,n)$ with group size $n$, block size $k$, and index $\lambda$, is a triple $(X,\mathcal{A},\mathcal{B})$,
where $X$ is a set of $kn$ elements. The set $\mathcal{A}$ forms a partition of $X$
into $k$ classes, called \emph{groups}, each containing exactly $n$ elements.
Furthermore, $\mathcal{B}$ is a collection of $k$-subsets of $X$, referred to as
\emph{blocks}. These objects satisfy the property that any unordered pair of
distinct elements of $X$ is either contained in exactly one group or appears
together in precisely $\lambda$ blocks, but not both. In particular, when $\lambda = 1$, we simply write $TD(k, n)$. 
Transversal designs and their variants serve as building blocks for many algebraic and group-based constructions, including Latin squares and magic-type arrays. 
For comprehensive references on combinatorial designs, we refer the reader to Cameron \cite{zbMATH00050655}, Colbourn \cite{MR1370815, MR1435524},  Lindner \cite{MR2469212}, and Stinson \cite{MR2029249}.

\section{Modular Balanced Partition Design} 
Motivated by the study of integer partitions, constant-sum partitions of elements in additive Abelian groups, and combinatorial design theory, a new combinatorial structure, namely, a modular balanced partition design, is introduced and studied. Furthermore, the existence and structural properties of these designs are investigated with respect to various sets of integers and finite Abelian groups.
\begin{definition}\label{def1}
 Let $X$ be a finite set of integers with cardinality $\nu =\kappa \lambda$. A modular balanced partition design is a triple $(X, \mathcal{A}, \mathcal{B})$, which satisfies the following conditions.  
\begin{enumerate}
    \item[$(i)$] $\mathcal{A}$ partitions $X$ into $\kappa$ blocks of size $\lambda$, such that every element in $X$ appears in exactly one block and whenever $\mathcal{A} = \{A_1, A_2, \ldots, A_{\kappa}\}$, then
        \(
        \sum_{a\in A_i} a \equiv i \lambda \bmod{\nu},   i \in [1, \kappa].\)
        \item[$(ii)$] $\mathcal{B}$ partitions $X$ into $\lambda$ blocks of size $\kappa$, such that every element in $X$ appears in exactly one block and whenever $\mathcal{B} = \{B_1, B_2, \ldots, B_{\lambda}\}$, then
        \(
        \sum_{b\in B_j} b \equiv j \kappa \bmod{\nu},  j \in [1, \lambda].\)
       \item[$(iii)$]  $A_i\cap B_j$ has  exactly one element, where $A_i\in \mathcal{A}$ and  $B_j \in \mathcal{B}$. 
\end{enumerate} 
\end{definition}
\noindent A \textit{modular balanced partition design} is shortly termed as ${(\nu,\kappa,\lambda)\text{-}\MBPD}$, where $\nu= |X|, \kappa=|\mathcal{A}|,$ and $\lambda=|\mathcal{B}|$.  Modular conditions on the sums of elements in the blocks ensure that the partitions exhibit cyclic and symmetric properties, which are highly desirable in many applications. Note that if $X\subset \mathbb{Z}$ has $\nu$-elements, we use addition modulo $\nu$ and similarly, if $X=(\Gamma, +)$, where $(\Gamma, +)$ is a finite additive group, we use binary operation $+$ to verify the condition (ii) and (ii) in the above definition. \\
Now, we give two examples in which the first one has $X$ as the set of 15 integers and in the second one,  $X$ is the cyclic group of order 16. 
\begin{example} Consider a $(15, 3,5)$-modular balanced partition design, where \begin{align*}
    X&  = \{-13, -12, -3, 0, 1, 2, 4, 5, 6, 7, 8, 9 , 13, 16, 17\} \\
    \mathcal{A}  & =\{\{-12, -3, 0, 6, 9\}, \{2, 4, 5, 8, 16\}, \{-13, 1, 7, 13, 17\} \}\\
    \mathcal{B}  & = \{\{0,7,8\}, \{-3,1,5\}, \{-12,16,17\}, \{2,9,13\}, \{-13,4,6\}\}.
\end{align*}

\end{example}
\begin{example} Consider a  $(16, 4,4)$-modular balanced partition design, where  \begin{align*}
    X&  =(\mathbb{Z}_{16}, \oplus_{16}),\\ 
    \mathcal{A} &  =\{\{1,2,3,10\}, \{5,0,4,11\},\{8,13,7,12\},\{6,9,14,15\}\}\\
    \mathcal{B} &  = \{\{10,11,15,12\}, \{1,5,6,8\},\{2,0,9,13\}, \{3,4,14,7\} \}.
\end{align*}
\end{example}
The above structure can be visualised as a grid or affine plane, where the rows represent blocks from $\mathcal{A}$ and the columns represent blocks from $\mathcal{B}$, in which  each element in $(i,j)$-th cell is the unique element of $A_i\cap B_j$.
Similarly, if $X=\mathbb{Z}_{36}$, then Figure \ref{fig:7.1}, represents an $(36,6,6)$-modular balanced partition design.
\begin{figure}[h]
    \centering
   \begin{tikzpicture}[scale=1]
       \node at (0,0){
\begin{tabular}{|c|c|c|c|c|c|}
\hline
15 & 22 & 9  & 14 & 26 & 28\\\hline
18 & 8  & 1  & 21 & 32 & 4\\\hline
25 & 3  & 16 & 34 & 2  & 10\\\hline
20 & 27 & 11 & 33 & 24 & 17\\\hline
6  & 29 & 12 & 7  & 35 & 13\\\hline
30 & 31 & 5  & 23 & 19 & 0\\ \hline
\end{tabular}};
  \end{tikzpicture}
    \caption{(36,6,6)\text{-}\MBPD }
    \label{fig:7.1}
\end{figure}

\begin{note}
     In a $(\nu, \kappa, \lambda)$-$\MBPD$, the total number of distinct elements is $\kappa\lambda$, the total number of blocks is $\kappa+\lambda$, and every element in $X$ occurs in exactly $2$ blocks.
\end{note}
\begin{note}
    If $\kappa=\lambda$, then the set of all collection of blocks in modular balanced partition designs $ \mathcal{A}\cup \mathcal{B}$ is contained in a collection of blocks of a Steiner $\kappa$-system $S(1,\kappa, \nu)$.  
\end{note}

\noindent For a  modular balanced partition design $(X, \mathcal{A}, \mathcal{B})$ with \(\mathcal{A}=\{A_i: i \in [1,\kappa]\}\) and \(\mathcal{B}=\{B_j: j \in [1,\lambda]\}\). We use 
\[
\mathcal{A}_{\mathrm{sum}}
= \sum_{i=1}^{\kappa}\sum_{a\in A_i} a,
\qquad
\mathcal{B}_{\mathrm{sum}}
= \sum_{j=1}^{\lambda}\sum_{b\in B_j} b,
\]
and
\[
\Omega_{\mathcal{A}}
= \biggl\{\sum_{a\in A_i} a : i\in [1,\kappa]\biggr\},
\qquad
\Omega_{\mathcal{B}}
= \biggl\{\sum_{b\in B_j} b : j\in [1,\lambda]\biggr\}.
\]
Thus, \(\mathcal{A}_{\mathrm{sum}}\) and \(\mathcal{B}_{\mathrm{sum}}\) are the total sums of all elements in all the blocks of \(\mathcal{A}\) and \(\mathcal{B}\), respectively, while \(\Omega_{A}\) and \(\Omega_{B}\) are the set of all sums of all the blocks of $\mathcal{A}$ and $\mathcal{B}$, respectively.
From Definition \ref{def1}, we obtain the following propositions,
\begin{proposition}{\label{lemma1}}
For any modular balanced partition design $(X, \mathcal{A}, \mathcal{B})$,  
$\sum_{x\in X}x=\mathcal{A}_{\mathrm{sum}}=\mathcal{B}_{\mathrm{sum}}$.
\end{proposition}
\begin{proof}
   The proof follows from the fact that each integer contributes exactly once to \(\mathcal{A}_{\mathrm{sum}}\) and exactly once to \(\mathcal{B}_{\mathrm{sum}}\).
\end{proof} 
The converse of the above Proposition \ref{lemma1} need not be true. 
\begin{proposition}
\label{thmim} Let $X=\mathbb{Z}_\nu$. If $(X, \mathcal{A}, \mathcal{B})$ is a $(\nu,\kappa,\lambda)$-$\MBPD$, then $ \Omega_{\mathcal{A}}$ and $ \Omega_{\mathcal{B}}$ are two subgroups of order $\kappa$ and $\lambda$, respectively.
\end{proposition}
\begin{proof}
    If \(\mathcal{A}=\{A_i: i \in [1,\kappa]\}\) and \(\mathcal{B}=\{B_j: j \in [1,\lambda]\}\), then $\Omega_{\mathcal{A}}=\{i\lambda: i \in [0,\kappa-1]\}$ and $\Omega_{\mathcal{B}} =\{j\kappa: j \in [0,\lambda-1]\}$. Then, $\Omega_{\mathcal{A}}$ and $\Omega_{\mathcal{B}}$ are two cyclic subgroups of $\mathbb{Z}_{\nu}$ generated by $\lambda$ and $\kappa$, respectively.
\end{proof}
%\subsection{Necessary conditions for   modular balanced partition designs}
\noindent Now, we proceed to prove the necessary conditions for the existence of a modular balanced partition design. 
\begin{theorem}{\label{thmnew}}
    If there exists a $(\nu, \kappa, \lambda)$-$\MBPD$, then $\kappa\equiv \lambda \bmod{2}$.
\end{theorem}
\begin{proof}
 Let $X$ be a finite set of integers with cardinality $\nu$, and let $\nu=\kappa \lambda$. Suppose \(\mathcal{A}=\{A_i: i \in [1,\kappa]\}\) and \(\mathcal{B}=\{B_j: j \in [1,\lambda]\}\).  Assume that there exists a $(\nu, \kappa, \lambda)$-$\MBPD$. Then we have \begin{align}
     \sum_{a\in A_i} a & \equiv i \lambda \bmod{\nu}, \quad i \in [1, \kappa],\\
     \sum_{b\in B_j} b & \equiv j \kappa \bmod{\nu}, \quad j \in [1, \lambda].
 \end{align}
      Since $\sum_{i=1}^{\kappa}\sum_{a\in A_i} a = \mathcal{A}_{\mathrm{sum}}$, $\sum_{j=1}^{\lambda}\sum_{b\in B_j} b= \mathcal{B}_{\mathrm{sum}}$ and by  Proposition \ref{lemma1}, we have, 
      \begin{align}\label{eqn3}
         \sum_{x\in X}x & \equiv \frac{\kappa\lambda(\lambda+1)}{2}\bmod{\kappa\lambda} \textnormal{~,~and~}  \sum_{x\in X}x \equiv \frac{\kappa\lambda(\kappa+1)}{2}\bmod{\kappa\lambda}. 
      \end{align} 
Therefore, \eqref{eqn3}, gives  \begin{align}\label{eqn4}
    \frac{\kappa\lambda(\kappa+1)}{2} & \equiv \frac{\kappa\lambda(\lambda+1)}{2}\bmod{\kappa\lambda}. 
\end{align}
From \eqref{eqn4}, $\frac{1}{2}(\kappa-\lambda)$ is an integer, and hence $\kappa$ and $\lambda$ are either both odd or both even.
\end{proof}
%-----------------
 \begin{corollary}\label{c} If $\nu \equiv 2 \bmod{4}$, then there is no $(\nu, \kappa, \lambda)$-$\MBPD$.
\end{corollary}
 \begin{proof}
   Suppose if there exists a $(\nu, \kappa, \lambda)$-$\MBPD$ with $\nu \equiv 2 \bmod{4}$, then $\nu=4k+2=\kappa\lambda$. But both $\kappa$ and  $\lambda$ can not be even. \end{proof}
\begin{theorem}{\label{x}}
    If $(\nu, \kappa, \lambda)$-$\MBPD$ exists, then\[\sum_{x\in X}x\equiv \begin{cases}
        \frac{\kappa\lambda}{2}\bmod{\kappa\lambda} & \textnormal{~if~} \nu \textnormal{~is even},\\
        0\bmod{\kappa\lambda} & \textnormal{~if~} \nu \textnormal{~is odd}.
    \end{cases}\]
\end{theorem}
\begin{proof}
    Assume that there exists a modular balanced partition design $(\nu, \kappa, \lambda)$-$\MBPD$. Let $X$ be a set of integers with cardinality $\nu$. By Proposition \ref{lemma1}, we have, 
    \begin{align}
       \sum_{x\in X}x & = \mathcal{A}_{\mathrm{sum}}\equiv(\lambda+2\lambda+\cdots+(\kappa-1)\lambda+\kappa\lambda)\bmod{\kappa\lambda},~~\textnormal{~and~}\\   
         & = \mathcal{B}_{\mathrm{sum}}\equiv (\kappa+2\kappa+\cdots+(\lambda-1)\kappa+\lambda\kappa)\bmod{\kappa\lambda}. \end{align}
    \noindent Case (i). If $\nu$ is even, then by Theorem \ref{thmnew}, both $\kappa$ and $\lambda$ are even. Therefore, 
    \begin{align*}
        \mathcal{A}_{\mathrm{sum}}& \equiv \bigl(\lambda-\lambda+2\lambda-2\lambda+\cdots+ \bigl(\kappa/2-1\bigr)\lambda-\bigl(\kappa/2-1\bigr)\lambda+\bigl(\kappa/2\bigr)\lambda\bigr)\bmod{\kappa\lambda}\equiv \frac{\kappa\lambda}{2}\bmod{\kappa\lambda}, \textnormal{~and~}\\
 \mathcal{B}_{\mathrm{sum}}& \equiv \bigl(\kappa-\kappa+2\kappa-2\kappa+\cdots+ \left(\lambda/2-1\right)\kappa-\left(\lambda/2-1\right)\kappa+\left(\lambda/2\right)\kappa\bigr)\bmod{\lambda\kappa}\equiv \frac{\kappa\lambda}{2}\bmod{\kappa\lambda}.
 \end{align*}
 \noindent Case (ii). If $\nu$ is odd, then by Theorem \ref{thmnew}, both $\kappa$ and $\lambda$ are odd. Therefore, 
 \begin{align*}
 \mathcal{A}_{\mathrm{sum}}& \equiv \bigl(\lambda-\lambda+2\lambda-2\lambda+\cdots+ \frac{\kappa+1}{2}\lambda-\frac{\kappa+1}{2}\lambda\bigr)\bmod{\lambda\kappa}\equiv 0\bmod{\kappa\lambda}, \textnormal{~and~} \\
\mathcal{B}_{\mathrm{sum}}& \equiv \bigl(\kappa-\kappa+2\kappa-2\kappa+\cdots+ \frac{\lambda+1}{2}\kappa-\frac{\lambda+1}{2}\kappa\bigr)\bmod{\lambda\kappa}\equiv 0\bmod{\kappa\lambda}       
    \end{align*}
    This completes the proof. 
\end{proof}
%------------------
\begin{corollary}
    If $(\nu,\kappa,\lambda)$-\MBPD ~exists with $\nu\equiv 0 \bmod 4$, then $\displaystyle\sum_{x\in X}x\equiv 0\bmod 4$.
\end{corollary}
\begin{proof}
    Suppose that there exists a $(\nu, \kappa, \lambda)$-$\MBPD$  with $\nu \equiv 0 \bmod{4}$. Then $\kappa\lambda \equiv 0 \bmod{2}$ and $\frac{\kappa\lambda}{2}\equiv 0 \bmod{2}$, which implies $\frac{\kappa\lambda}{2}\equiv 0 \bmod{4}$ and consequently, $\displaystyle\sum_{x\in X}x\equiv 0 \bmod{4}$. 
\end{proof}
%-----------------
\begin{corollary}
    If $X=\{\pm x : x\in \mathbb{Z}\setminus\{0\}\}$, then there is no $(\nu, \kappa, \lambda)$-\MBPD. 
\end{corollary}
%-----------------

\begin{theorem}
Suppose $(X, \mathcal{A}, \mathcal{B})$ is a $(\nu,\kappa,\lambda)$-\MBPD.  
For any $x\in \mathbb{Z}\setminus\{0\}$, let $X' = \{\,x+y : y\in X\,\}, \mathcal{A}' =\{A_i': i \in [1, \kappa]\}$ and
$\mathcal{B}' = \{B_j': j \in [1, \lambda]\}$, where  $A'_i = \{x+y : y\in A_i\}$ and  $B'_j = \{x+y : y\in B_j\}$. Then $(X', \mathcal{A}', \mathcal{B}')$ is also a $(\nu,\kappa,\lambda)\text{-}\MBPD$.
\end{theorem}

\begin{proof}Since each element of $X$ is shifted by adding the same integer $x$.  Therefore, 
\(
|X'| = |X| = \nu.
\)
Also, 
for any block $A\in \mathcal{A}$, the shifted block $A'=\{x+y : y\in A\}$ has the same size as $A$.  
Thus, each block in $\mathcal{A}'$ has the same size as in $\mathcal{A}$, and similarly for $\mathcal{B}'$. Moreover, the
disjoint blocks in $\mathcal{A}$ and $\mathcal{B}$ are remain disjoint after adding the same integer to all elements. Therefore, 
the union of all shifted block is in $\mathcal{A}$ and $\mathcal{B}$ are $\{x+y : y\in X\} = X'$. 
If $y_1,y_2\in X$ occur together in certain blocks, then after shifting, the pair $(x+y_1, x+y_2)$ occurs in the corresponding shifted blocks.  
Hence, the pairwise balance condition is preserved. Therefore, for $\mathcal{A}'=\{A_i': i \in [1, \kappa]\}$, we have \begin{align*}
    \sum_{a\in A_{i}'}a & = \lambda x \bmod{\nu}+i\lambda \bmod{\nu}=\lambda(x+i)\bmod{\nu} =\lambda i \bmod \nu.
\end{align*}
since adding $\lambda x$ modulo $\nu$ permutes the set of multiples of $\lambda$. Then, $\Omega_{\mathcal{A}'}=\{\lambda i \bmod \nu: i \in [1, \kappa]\}$. 
Analogously for $\mathcal{B}'$,  \begin{align*}
    \sum_{b\in B_{j}'}b & = \kappa x \bmod{\nu}+j\kappa \bmod{\nu}=\kappa(x+j)\bmod{\nu} = \kappa j\bmod{\nu}.
\end{align*}
Also, $\Omega_{\mathcal{B}'}=\{\kappa j\bmod{\nu}: j \in [1, \lambda]\}$. It is easy to check that $\Omega_{\mathcal{A}'}$ and $\Omega_{\mathcal{B}'}$ are subgroups of $\mathbb{Z}_{\nu}$ with
$|\Omega_{\mathcal{A}'}|=\kappa$ and $|\Omega_{\mathcal{B}'}|=\lambda$.
Therefore,  $(X', \mathcal{A}', \mathcal{B}')$ is also a $(\nu,\kappa,\lambda)$-MBPD.
\end{proof}
\begin{corollary}
    If there exists a $(\nu, \kappa, \lambda)$-$\MBPD$ over $X$, then there exists a  $(\nu, \kappa, \lambda)$-$\MBPD$ over $X'$, where $X'=\{-x: x\in X\}$.
\end{corollary}

% \begin{theorem}
%     Suppose $X$ and $\mathcal{Y}$ are two disjoint sets with the same cardinality $\nu$. If there exists $(\nu, \kappa, \lambda)$-\MBPD~ over $X$ and $\mathcal{Y}$, then there exists a $(2\nu, 2\kappa, \lambda)$-\MBPD~over $X\cup \mathcal{Y}$.
% \end{theorem}
% \begin{proof}
    
% \end{proof}
\begin{theorem}Let $X=\mathbb{Z}_{\nu}$. If there exists a $(\nu,\kappa,\lambda)$-\MBPD, then for every $\alpha\in\mathbb{Z}^{+}$, there exists an $(\alpha\nu,\kappa,\alpha\lambda)$-\MBPD.
\end{theorem}
\begin{proof} 
Assume that $(X,\mathcal{A},\mathcal{B})$ is a $(\nu,\kappa,\lambda)$-\MBPD, where
\(X=\mathbb{Z}_{\nu}\).
We shall construct a $(\alpha\nu,\kappa,\alpha\lambda)$-\MBPD   ~over
\(X'=\mathbb{Z}_{\alpha\nu}\) for an arbitrary
\(\alpha\in\mathbb{Z}^{+}\). Let
\(
\mathcal{A}=\{A_i:i\in[1,\kappa]\}
~\text{and}~
\mathcal{B}=\{B_j:j\in[1,\lambda]\}.
\)
\noindent For each \(i\in[1,\kappa]\), let
\(
A_i'
=
\bigl\{
t\nu+y \pmod{\alpha\nu}
:\,
y\in A_i,\;
t\in[0,\alpha-1]
\bigr\},
\)
and set
\(
\mathcal{A}'
=
\{A_i':i\in[1,\kappa]\}.
\)
Similarly, for each \(j\in[1,\lambda]\) and \(t\in[0,\alpha-1]\), define
\(
B_{j,t}'
=
\bigl\{
t\nu+y \pmod{\alpha\nu}
:\,
y\in B_j
\bigr\},
\)
and let
\(
\mathcal{B}'
=
\{B_{j,t}':(j,t)\in[1,\lambda]\times[0,\alpha-1]\}.
\)
Since \(|A_i|=\lambda\) and \(|B_j|=\kappa\), it follows that
\(
|A_i'|=\alpha\lambda
~\text{for all }i\in[1,\kappa],
\)
and
\(
|B_{j,t}'|=\kappa
~\text{for all }(j,t)\in[1,\lambda]\times[0,\alpha-1].
\)
\noindent Now, we need to verify the modular conditions for both the sets in $\mathcal{A}'$ and $\mathcal{B}'$. 
 That is, \begin{align*}
    \sum_{a\in A_{i}'}a & = \lambda\nu \bmod{\alpha\nu}+\lambda 2 \nu\bmod{\alpha\nu}+ \lambda(\alpha-1)\bmod{\alpha\nu + \lambda i \alpha\bmod{\nu}}\\
     & = \lambda\sum_{t=0}^{\alpha-1}t\nu \bmod{\alpha\nu}+\lambda i \alpha\bmod{\nu}\\
     & = \lambda\nu\frac{\alpha(\alpha-1)}{2}\bmod{\alpha \nu} + \lambda i \alpha\bmod{\nu} = \lambda i \alpha \bmod{\nu}\equiv \lambda i \alpha \bmod{\alpha \nu}. 
     \end{align*}
\noindent Therefore,
\(
\Omega_{\mathcal{A}'}
=
\{\alpha\lambda i \pmod{\alpha\nu}: i\in[1,\kappa]\}.
\)
Since
\(
\Omega_{\mathcal{A}}
=
\{\lambda i \pmod{\nu}: i\in[1,\kappa]\},
\)
and the elements \(\lambda,2\lambda,\ldots,\kappa\lambda\) are pairwise distinct modulo \(\nu\), it follows that
\(
\alpha\lambda,\;2\alpha\lambda,\;\ldots,\;\kappa\alpha\lambda
\)
are pairwise distinct modulo \(\alpha\nu\). Consequently,
\(\Omega_{\mathcal{A}'}\) contains exactly \(\kappa\) elements. Moreover, as it is generated by \(\alpha\lambda\) in \(\mathbb{Z}_{\alpha\nu}\), \(\Omega_{\mathcal{A}'}\) forms a subgroup of \(\mathbb{Z}_{\alpha\nu}\) of order \(\kappa\). Now, for the set $\mathcal{B}',$
\begin{align*}
    \sum_{b\in B_{j,t}'}b& = \kappa t \nu \bmod{\alpha\nu}+\kappa j\bmod{\nu} = \kappa^2t{\lambda} \bmod{\alpha\nu}+\kappa j \bmod{\nu} = \kappa j'\bmod\alpha{\nu}.
\end{align*} Therefore, $\Omega_{\mathcal{B}'}=\{\kappa j'\bmod\alpha{\nu}:j'\in[1, \alpha\lambda]$, since the elements
\(
\kappa,\;2\kappa,\;\ldots,\;\alpha\lambda\kappa
\)
are distinct modulo \(\alpha\nu\), it follows that \(\Omega_{\mathcal{B}'}\) contains exactly \(\alpha\lambda\) elements. Hence, \(\Omega_{\mathcal{B}'}\) is a subgroup of \(\mathbb{Z}_{\alpha\nu}\) of order \(\alpha\lambda\)..  Therefore, the proof is complete. 
\end{proof}

\section{Construction of \MBPD~ via subgroup magic rectangle}
In this section, a modular balanced partition design is constructed by using a subgroup magic rectangle.
\begin{lemma}\label{lemmamain}
Let $X=\mathbb{Z}_\nu$ with $\nu=\kappa\lambda$. 
Then a $(\nu,\kappa,\lambda)$-modular balanced partition design 
$(X,\mathcal{A},\mathcal{B})$ exists if and only if  
a $\kappa\times \lambda$ subgroup magic rectangle over $\mathbb{Z}_\nu$ exists. 
\end{lemma}

\begin{proof}
\smallskip
 % If an $\MBPD$ exists, then a $\SMR$ over $\mathbb{Z}_v$ exists}.
\par Suppose that $(X,\mathcal{A},\mathcal{B})$ is a 
$(\nu,\kappa,\lambda)$-$\MBPD$ with $X=\mathbb{Z}_\nu$.
By definition, $\mathcal{A}=\{A_i : i \in [1,\kappa]\}$ is a partition of 
$\mathbb{Z}_\nu$ into $\kappa$ blocks of size $\lambda$, and 
$\mathcal{B}=\{B_j: j \in [1, \lambda]\}$ is a partition into $\lambda$ blocks 
of size $\kappa$. Since condition (iii) guarantees that 
$|A_i\cap B_j|=1$ for $i$ and$j$, construct a 
$\kappa\times\lambda$ array, whose $(i,j)$-th entry is  $\zeta_{ij}$, where  
$A_i\cap B_j=\{\zeta_{ij}\}$. Note that each element $\zeta_{ij}$ of $\mathbb{Z}_\nu$ 
appears exactly once in this array. Now, 
\[
h_i=\sum_{j=1}^\lambda \zeta_{ij}
    = \sum_{x\in A_i} x \bmod{\nu}, 
\qquad 
h'_j=\sum_{i=1}^\kappa \zeta_{ij}
    = \sum_{x\in B_j} x \bmod{\nu}.
\]
From the MBPD conditions,
\[
\sum_{a\in A_i}a \equiv i\lambda \bmod{\nu}, \quad
\sum_{b\in B_j}b \equiv j\kappa \bmod{\nu}.
\]
Hence $h_i \equiv i\lambda \bmod{\nu}$ and 
$h'_j \equiv j\kappa \bmod{\nu}$. Therefore
\[
H_1=\{h_i: i \in [1,\kappa]\}
=\{\,i\lambda \bmod{\nu} : i\in [1,\kappa]\},
\]
\[
H_2=\{h'_j: j\in[1,\lambda]\}
=\{\,j\kappa \bmod{\nu} : j\in[1,\lambda]\}.
\]
For any $h_k,h_\ell \in H_1$, we have
$h_k-h_\ell \equiv (k-\ell)\lambda \bmod{\nu} \in H_1$, 
and similarly for $H_2$. Thus, the constructed array is an 
$\SMR_{\mathbb{Z}_\nu}(\kappa,\lambda:H_1,H_2)$. The forward direction is proved. 
 Conversely, assume we have an 
$\SMR_{\mathbb{Z}_\nu}(\kappa,\lambda:H_1,H_2)$, 
a $\kappa\times\lambda$ array containing each element of $\mathbb{Z}_\nu$ 
exactly once, with row-sums $H_1=\{h_i: i \in [1,\kappa]\}$ 
and column-sums $H_2=\{h'_j: j\in[1,\lambda]\}$.
The closure conditions imply that $H_1$ and $H_2$ 
are subgroups of $\mathbb{Z}_\nu$. Since $\nu=\kappa\lambda$ and 
$\mathbb{Z}_\nu$ is cyclic, these subgroups must be
\[
H_1 = \langle \lambda \rangle = 
\{\,t\lambda : t\in [0,\kappa-1]\,\}, \qquad
H_2 = \langle \kappa \rangle = 
\{\,s\kappa : s\in[0,\lambda-1]\,\}.
\]
After re-indexing rows and columns, let 
$h_i\equiv i\lambda \bmod{\nu}$ and 
$h'_j\equiv j\kappa \bmod{\nu}$. \\
For $i\in [1,\kappa]$ and $j\in[ 1, \lambda]$, define  $A_i$ as the set of entries in $i$-th row, and $B_j$ as the set of 
entries in $j$-th column. Then $\mathcal{A}=\{A_i: i\in [1,\kappa]\}$ partitions $\mathbb{Z}_\nu$ 
into $\kappa$ subsets of size $\lambda$, and $\mathcal{B}=\{B_j:j\in[ 1, \lambda]\}$ partitions 
$\mathbb{Z}_\nu$ into $\lambda$ subsets of size $\kappa$.  
Since each cell of the array contains exactly one 
element, $|A_i\cap B_j|=1$. Finally,
\[
\sum_{a\in A_i}a \equiv h_i \equiv i\lambda \bmod{\nu}, \textnormal{~and~}
\sum_{b\in B_j}b \equiv h'_j \equiv j\kappa \bmod{\nu}.
\]
Thus $(\mathbb{Z}_\nu,\mathcal{A},\mathcal{B})$ is a 
$(\nu,\kappa,\lambda)$-$\MBPD$.

\smallskip
\noindent Therefore, the two  combinatorial objects are equivalent.
\end{proof}

    \begin{theorem}{\label{thmfre}}
    If $\nu$ is odd, then there exists a $(\nu, \kappa, \lambda)$-$\MBPD$~over the group $\mathbb{Z}_{\nu}$.
\end{theorem}
\begin{proof}
   The proof follows from [\cite{karthik2025existence}, Lemma 2],  and [\cite{karthik2025existence}, Theorem 6].
\end{proof}
\begin{theorem}{\label{thmkar}}
      If $n\equiv  0 \bmod{4}$, then there exists a $(n^2,n ,n)$-$\MBPD$~over $\mathbb{Z}_{n^2}$.
\end{theorem}

 \begin{proof}
     From Lemma \ref{lemmamain}, it is enough to construct an $n\times n$ subgroup magic rectangle over the group $\mathbb{Z}_{n^2}$. The construction is as follows:\\
     \noindent Step 1. First, we construct an $n\times n$ matrix $\mathbb{M}=(\zeta_{i,j})$ with $\zeta_{i,j} \in \mathbb{Z}_{n^2}$, is defined by 
    \[\zeta_{i,j}=\begin{cases}
        (n(j-1)+i)\bmod{n^2} & \textnormal{if~} j\in [1,n] \textnormal{~and} ~j~ \text{is odd},\\
        (nj-(i-1))\bmod{n^2} & \textnormal{if~} j\in [1,n] \textnormal{~and} ~j~ \text{is even}.
    \end{cases}\]
Then, for each $i\in [1,n]$, the row-sums of $\mathbb{M}$ is given below. 
\begin{align*}
    \sum_{j=1}^{n}\zeta_{i,j}=\frac{n^2(n^2+1)}{2n}\pmod{n^2}=\frac{n}{2}\pmod{n^2}.
\end{align*}
Similarly, for each $j\in[1,n]$, the column-sums of $\mathbb{M}$ is, 
\begin{align*}
    \sum_{i=1}^{n}\zeta_{i,j}=\frac{n(n+1)}{2}\pmod{n^2}.
\end{align*}Note that the above step is the first step for the construction of an even order magic square (see \cite{MR2569784}). At this stage, the set of all $n$ row-sums (or column-sums) does not form a subgroup of $\mathbb{Z}_{n^2}$. Further rearrangements in  $\mathbb{M}$ will be done in next steps so that the set of all row-sums and column-sums form subgroups.  
\vskip .1cm \noindent Step 2. Now, we define  $\mathbb{M}'=(\zeta_{i',j'})$ from $\mathbb{M}$  by swapping only the elements $(i,j)\rightleftharpoons (\frac{n}{2}-i+1, n-j+1)$, whenever  $i=j$ and $j\in [1, \frac{n}{4}]$ and swap the elements $(i,j)\rightleftharpoons (n-j+1, n-j+1)$, whenever $i=\frac{n}{2}+j$, where $j\in [1, \frac{n}{4}]$ and fix the remaining elements as the same. That is, \[
\zeta_{i',j'} =
\begin{cases}
   \zeta_{\tfrac{n}{2}-j+1,\, n-j+1}, & \text{if } i=j,\; j \in \big[1, \tfrac{n}{4}\big], \\[6pt]
    \zeta_{n-j+1,\, n-j+1}, & \text{if } i=j-\tfrac{n}{2} ,\; j \in \big[n-\tfrac{n}{4}+1, n\big], \\[6pt]
   \zeta_{\,n-j+1,\, n-j+1}, & \text{if } i=\tfrac{n}{2}+j,\; j \in \big[1, \tfrac{n}{4}\big], \\[6pt]
    \zeta_{-\frac{n}{2}+j+1,\, n-j+1}, & \text{if } i=j ,\; j \in \big[n-\tfrac{n}{4}+1, n\big], \\[6pt]
   \zeta_{i,j}, & \text{otherwise}.
\end{cases}
\]
\noindent In this step, the sum of the entries in each row is determined. 

\noindent For each $i\in [1, \frac{n}{4}]$, the $i$-th row-sum $\Sigma[R_i]$  of $\mathbb{M}'$ is given by,
 \[\Sigma[R_i] = \sum_{j=1}^{n}\zeta_{i,j}-\zeta_{i,i}+\zeta_{\frac{n}{2}-(i-1), n -(i-1)}. \numberthis \label{7}\]
\begin{align*}
  \Sigma[R_i] & = \begin{cases}
        \frac{n}{2}\bmod{n^2}-(n(i-1)+i)\bmod{n^2}+(n(n-(i-1))-(\frac{n}{2}-(i-1)-1)\bmod{n^2} & \text{if}~ i ~ \text{is odd}\\
\frac{n}{2}\bmod{n^2}-(ni-(i-1))\bmod{n^2}+n(n-(i-1)-1)\bmod{n^2}+(\frac{n}{2}-(i-1))\bmod{n^2}& \text{if}~ i ~ \text{is even}\\
    \end{cases}\\[6pt]
    & =\begin{cases}
        \left(\frac{n}{2}-(ni-n+i)+(n^2-ni+n-\frac{n}{2}+i\right)\bmod{n^2}& \text{if}~ i ~ \text{is odd}\\
 \left(\frac{n}{2}-ni+i-1+n^2-ni=\frac{n}{2}-i+1\right)\bmod{n^2}& \text{if}~ i ~ \text{is even}\\
    \end{cases}\\[6pt]
    & = \begin{cases}
       (-2ni+2n)\bmod{n^2}& \text{if}~ i ~ \text{is odd}\\
    (-2ni+n)\bmod{n^2}& \text{if}~ i ~ \text{is even}.
    \end{cases}
\end{align*}
Now for each $i\in [\frac{n}{4}+1, \frac{n}{2}]$, the row-sum $\Sigma[R_i]$ is given as, 
\begin{align*}
    \Sigma[R_i]&= \sum_{j=1}^{n}\zeta_{i,j}-\zeta_{i, \frac{n}{2}+i}+\zeta_{\frac{n}{2}-i+1,\frac{n}{2}-i+1}\numberthis.
\end{align*}
\begin{align*}
    \Sigma[R_i] & = \begin{cases}\frac{n}{2}\bmod{n^2}-\left(\frac{n^2}{2}+ni-n+i\right)\bmod{n^2} +\left(\frac{n^2}{2}-ni+n-\frac{n}{2}+i\right)\bmod{n^2}  & \text{if}~ i ~ \text{is odd}\\
    \frac{n}{2}\bmod{n^2}-\left(\frac{n^2}{2}+ni-i+1\right)\bmod{n^2} +\left(\frac{n^2}{2}-ni+\frac{n}{2}-i+1\right)\bmod{n^2}  & \text{if}~ i ~ \text{is even}\\
    \end{cases}\\[6pt]
    &= \begin{cases}\left(\frac{n}{2}-\frac{n^2}{2}-ni+n-i +\frac{n^2}{2}-ni+n-\frac{n}{2}+i\right)\bmod{n^2}  & \text{if}~ i ~ \text{is odd}\\
    \left(\frac{n}{2}-\frac{n^2}{2}-ni+i +\frac{n^2}{2}-ni+\frac{n}{2}-i\right)\bmod{n^2}  & \text{if}~ i ~ \text{is even}\\
    \end{cases}
    \\[6pt]
    &= \begin{cases}\left(-2ni+2n\right)\bmod{n^2}  & \text{if}~ i ~ \text{is odd}\\
    \left(-2ni+n\right)\bmod{n^2}  & \text{if}~ i ~ \text{is even}.\\
    \end{cases}
    \end{align*}
Further, for each $i\in [\frac{n}{2}+1, \frac{n}{2}+\frac{n}{4}]$, the $i$-th row-sum, $\Sigma[R_i]$ is given by, 
\[\Sigma[R_i]=\sum_{j=1}^{n}\zeta_{i,j}- \zeta_{i, i-\frac{n}{2}} + \zeta_{\frac{3n}{2}-i+1,\frac{3n}{2}-i+1 }. \numberthis\]
\begin{align*}
    \Sigma[R_i] &  = \begin{cases}\frac{n}{2}\bmod{n^2}-\left(ni-\frac{n^2}{2}-n+i\right)\bmod{n^2} +\left(\frac{3n^2}{2}-ni+n-i+1\right)\bmod{n^2}  & \text{if}~ i ~ \text{is odd}\\
    \frac{n}{2}\bmod{n^2}-\left(ni-\frac{n^2}{2}-i+1\right)\bmod{n^2} +\left(\frac{3n^2}{2}-ni+n+i\right)\bmod{n^2}  & \text{if}~ i ~ \text{is even}\\
    \end{cases}\\[6pt]
    &= \begin{cases}\left(\frac{n}{2}-ni+\frac{n^2}{2}+n-i +\frac{3n^2}{2}-ni+n-(\frac{3n}{2}-i+1)+1\right)\bmod{n^2}  & \text{if}~ i ~ \text{is odd}\\
    \left(\frac{n}{2}-ni+\frac{n^2}{2}+(\frac{3n}{2}-i+1) -1 +\frac{3n^2}{2}-ni+n+i\right)\bmod{n^2}  & \text{if}~ i ~ \text{is even}\\
    \end{cases}
    \\[6pt]
    &= \begin{cases}\left(-2ni+n\right)\bmod{n^2}  & \text{if}~ i ~ \text{is odd}\\
    \left(-2ni+2n\right)\bmod{n^2}  & \text{if}~ i ~ \text{is even}.\\
    \end{cases}
\end{align*}
Similarly, for each $i\in [\frac{3n}{4}+1, n]$, the $i$-th row-sum $\Sigma[R_i]$ is given by, 
\[\Sigma[R_i]=\sum_{j=1}^{n}\zeta_{i,j}- \zeta_{i, i} + \zeta_{\frac{3n}{2}-i+1,n-i+1 } \numberthis\label{10}\]
\begin{align*}
    \Sigma[R_i] &  = \begin{cases}\frac{n}{2}\bmod{n^2}-\left(ni-n+i\right)\bmod{n^2} +\left(n^2-ni+n-i+1\right)\bmod{n^2}  & \text{if}~ i ~ \text{is odd}\\
    \frac{n}{2}\bmod{n^2}-\left(ni-(i-1)\right)\bmod{n^2} +\left(n^2-ni+i\right)\bmod{n^2}  & \text{if}~ i ~ \text{is even}\\
    \end{cases}\\[6pt]
    &= \begin{cases}\left(n-2ni\right)\bmod{n^2}  & \text{if}~ i ~ \text{is odd}\\
    \left(2n-2ni\right)\bmod{n^2}  & \text{if}~ i ~ \text{is even}.\\
    \end{cases}
\end{align*} Suppose $\mathcal{A}$ denotes the set of all row sums. We define
\[
\alpha(i)=
\begin{cases}
2n & i\text{ is odd and } i\in \left[1,\frac n2\right],
       \text{ or }
       i\text{ is even and } i\in \left[\frac n2+1,n\right],\\[4pt]
n & i\text{ is even and } i\in \left[1,\frac n2\right],
      \text{ or }
      i\text{ is odd and } i\in \left[\frac n2+1,n\right].
\end{cases}
\]
Then, by \eqref{7}--\eqref{10},
\[
\mathcal{A}
=
\left\{
(-2ni+\alpha(i)) \bmod n^{2}
:\;
i\in[1,n]
\right\}.
\]
\noindent Now, we need to prove that the set $\mathcal{A}$ is an $n$ order subgroup of $\mathbb{Z}_{n^2}$. Each element of $\mathcal{A}$ can be written as 
\[
a_i \equiv -2ni + \alpha(i) \bmod{n^2}, ~ \alpha(i)\in\{n,2n\},
\]
so that 
\[
a_i \equiv n \cdot t(i) \bmod{n^2}, ~
t(i) = -2i + \tfrac{\alpha(i)}{n}.
\]
Hence $\mathcal{A} \subseteq \langle n \rangle$. 
Moreover, the map $i \mapsto t(i) \bmod{n}$ produces all $n$ distinct residues 
(the two parity cases yield the $n/2$ odd and $n/2$ even residues). 
Thus $|\mathcal{A}| = n$, and therefore 
\(
\mathcal{A} = \langle n \rangle,
\)
which is a subgroup of $\mathbb{Z}_{n^2}$. At this stage, we have shown that the row-sums forms a subgroup. In the following step, we prove the the set of all column-sums forms a subgroup.
\vskip .1cm \noindent Step 3. At this stage, we proceed by applying the swap map to the columns in an analogous manner. Let $\mathbb{M}''=(\zeta_{i'',j''})$. Define  $\mathbb{M}''$ from $\mathbb{M}'$  by performing the swap $(i',j')\rightleftharpoons (i', i'+\frac{n}{2})$, where $i'\in [1, \frac{n}{2}]$, while keeping the remaining elements fixed. Therefore, \\
\noindent For each $j\in[1, \frac{n}{4}]$, the $j$-th column-sum $\Sigma[C_j]$, of $\mathbb{M}''$, is 
\[\Sigma[C_j]= \sum_{i=1}\zeta_{i,j}-(\zeta_{j,j}+\zeta_{j+\frac{n}{2}, j})+(\zeta_{j, \frac{n}{2}+j}+\zeta_{n-j+1, n-j+1}) \numberthis\label{11}\]
\begin{align*}
    \Sigma[C_j] & =\begin{cases}
        \frac{n(n+1)}{2}\bmod{n^2}-(2nj-2n+2j+\frac{n}{2}))\bmod{n^2}+(\frac{n^2}{2}+n^2+2j)\bmod{n^2} & \text{if}~ j ~ \text{is odd}\\ 
         \frac{n(n+1)}{2}\bmod{n^2}-(2nj-2j+\frac{n}{2}+2)\bmod{n^2}+(\frac{n^2}{2}-2j+2+n)\bmod{n^2} & \text{if}~ j ~ \text{is even}\\ 
    \end{cases}\\[6pt]
  & =\begin{cases}
        \left(\frac{n(n+1)}{2}-(2nj-2n+2j+\frac{n}{2})+(\frac{n^2}{2}+2j)\right)\bmod{n^2} & \text{if}~ j ~ \text{is odd}\\ 
        \left(\frac{n(n+1)}{2}-(2nj-2j+\frac{n}{2}+2)+(\frac{n^2}{2}-2j+2+n)\right)\bmod{n^2} & \text{if}~ j ~ \text{is even}\\ 
           \end{cases}\\[6pt]
           & =\begin{cases}
       (-2nj+2n)\bmod{n^2} & \text{if}~ j ~ \text{is odd}\\ 
        (-2nj+n )\bmod{n^2}& \text{if}~ j ~ \text{is even}.\\ 
           \end{cases}\\
\end{align*}
Now, for each $j\in [\frac{n}{4}+1, \frac{n}{2}]$, the $j$-th column-sum $\Sigma[C_j]$, is given below, 
\[\Sigma[C_j]= \sum_{i=1}\zeta_{i,j}-\zeta_{j,j}+\zeta_{\frac{n}{2}-j+1, \frac{n}{2}-j+1}\numberthis\]
\begin{align*}
     \Sigma[C_j] & =\begin{cases}
        \frac{n(n+1)}{2}\bmod{n^2}-(nj-n+j)\bmod{n^2}+(\frac{n^2}{2}-nj+n-\frac{n}{2}+j)\bmod{n^2} & \text{if}~ j ~ \text{is odd}\\ 
         \frac{n(n+1)}{2}\bmod{n^2}-(nj-j+1)\bmod{n^2}+(\frac{n^2}{2}-nj+\frac{n}{2}-j+1)\bmod{n^2} & \text{if}~ j ~ \text{is even}\\ 
    \end{cases}\\[6pt]
    & = \begin{cases}
        (\frac{n(n+1)}{2}-(nj-n+j)+(\frac{n^2}{2}-nj+n-\frac{n}{2}+j))\bmod{n^2} & \text{if}~ j ~ \text{is odd}\\ 
         (\frac{n(n+1)}{2}-(nj-j+1)+(\frac{n^2}{2}-nj+\frac{n}{2}-j+1))\bmod{n^2} & \text{if}~ j ~ \text{is even}\\ 
    \end{cases}\\[6pt]
    &  =\begin{cases}
       (-2nj+2n)\bmod{n^2} & \text{if}~ j ~ \text{is odd}\\ 
        (-2nj+n)\bmod{n^2} & \text{if}~ j ~ \text{is even}.\\ 
           \end{cases}\\
\end{align*}
Further, for each $j \in[\frac{n}{2}+1, \frac{3n}{4}]$, the $j$-th column-sum is, 
\[\Sigma[C_j]= \sum_{i=1}\zeta_{i,j}-\zeta_{j-\frac{n}{2}, j}+\zeta_{n-j+1,\frac{3n}{2} - j +1}\numberthis\]
\begin{align*}
     \Sigma[C_j] & =\begin{cases}
        \frac{n(n+1)}{2}\bmod{n^2}-(nj-n+j-\frac{n}{2})\bmod{n^2}+(\frac{3n^2}{2}-nj+j)\bmod{n^2} & \text{if}~ j ~ \text{is odd}\\ 
         \frac{n(n+1)}{2}\bmod{n^2}-(nj-j+\frac{n}{2}+1)\bmod{n^2}+(\frac{3n^2}{2}-nj+n-j+1)\bmod{n^2} & \text{if}~ j ~ \text{is even}\\ 
    \end{cases}\\[6pt]
    & = \begin{cases}
        \left(\frac{n(n+1)}{2}-(nj-n+j-\frac{n}{2})+(\frac{3n^2}{2}-nj+j)\right)\bmod{n^2} & \text{if}~ j ~ \text{is odd}\\ 
         \left(\frac{n(n+1)}{2}-(nj-j+\frac{n}{2}+1)+(\frac{3n^2}{2}-nj+n-j+1)\right)\bmod{n^2} & \text{if}~ j ~ \text{is even}\\ 
    \end{cases}\\[6pt]
    &  =\begin{cases}
       (-2nj +2n)\bmod{n^2}& \text{if}~ j ~ \text{is odd}\\ 
        (-2nj+n)\bmod{n^2} & \text{if}~ j ~ \text{is even}.\\ 
           \end{cases}\\
\end{align*}
Similarly,  for each $j\in[\frac{3n}{4}+1, n]$, the $j$-th column-sum is, \[\Sigma[C_j]= \sum_{i=1}\zeta_{i,j}-(\zeta_{j-\frac{n}{2}, j}+\zeta_{j,j})+\zeta_{j-\frac{n}{2},j- \frac{n}{2}}+\zeta_{\frac{3n}{2}-j+1, n-j+1}\numberthis \label{14}\]
\begin{align*}
     \Sigma[C_j] & =\begin{cases}
        \frac{n(n+1)}{2}\bmod{n^2}-(2nj-2n+2j-\frac{n}{2})\bmod{n^2}+(\frac{n^2}{2}+2j-n)\bmod{n^2} & \text{if}~ j ~ \text{is odd}\\ 
         \frac{n(n+1)}{2}\bmod{n^2}-(2nj-2j+\frac{n}{2}+2)\bmod{n^2}+(\frac{n^2}{2}-2j+n+2)\bmod{n^2} & \text{if}~ j ~ \text{is even}\\ 
    \end{cases}\\[6pt]
    & = \begin{cases}
        \left(\frac{n(n+1)}{2}-(2nj-2n+2j-\frac{n}{2})+(\frac{n^2}{2}+2j-n)\right)\bmod{n^2} & \text{if}~ j ~ \text{is odd}\\ 
         \left(\frac{n(n+1)}{2}-(2nj-2j+\frac{n}{2}+2)+(\frac{n^2}{2}-2j+n+2)\right)\bmod{n^2} & \text{if}~ j ~ \text{is even}\\ 
    \end{cases}\\[6pt]
    &  =\begin{cases}
       (-2nj +2n)\bmod{n^2}& \text{if}~ j ~ \text{is odd}\\ 
        (-2nj+n)\bmod{n^2} & \text{if}~ j ~ \text{is even}.\\ 
           \end{cases}\\
\end{align*}
 Let $\mathcal{B}$ be the set of all column sums. Define
$\beta(j)=
2n$ if $j$ is odd and 
$\beta(j)=n$ if $j$ is even. \\
Then from \eqref{11} to \eqref{14},
\[
\mathcal{B}
=
\left\{
(-2nj+\beta(j)) \bmod n^{2}
:\;
j\in[1,n]
\right\}.
\] \noindent Each element of $\mathcal{B}$ has the form 
$
b_j \equiv (-2nj + \beta(j))\bmod{n^2}, ~\beta(j)\in\{n,2n\},
$ and hence 
\[
b_i \equiv n t(j) \bmod{n^2}, 
\qquad t(j)=-2i+\tfrac{\beta(j)}{n}.
\]
Thus $\mathcal{B}\subseteq \langle n\rangle$. 
As $j$ ranges from $1$ to $n$, the values $t(j)\pmod n$ cover all $n$ residues 
(the odd $j$ give one parity class and the even $j$ give the other). 
Therefore $|\mathcal{B}|=n$, and we conclude that
\(
H_2 = \langle n \rangle,
\)
which is a subgroup of $\mathbb{Z}_{n^2}$.
 Consequently, there exists an $n \times n$ subgroup magic rectangle, $\SMR_{\mathbb{Z}_{n^2}}(n,n:H_1, H_2)$. 
\end{proof}
\begin{example}\label{ex} By  
Theorem $\ref{thmkar}$, the step-by-step construction of a $8\times 8$ subgroup magic rectangle over $\mathbb{Z}_{64}$ is given below. 
\end{example}
\noindent  \textbf{Step 1. Assignment of group elements to $8\times 8$ array.}  
We consider the group $\mathbb{Z}_{64}$ and assign its $64$ elements  
into an $8 \times 8$ array. \\  Define the array $\mathbb{M}=(\zeta_{i,j})$, whose  entries $\zeta_{i,j}$'s from $\mathbb{Z}_{64}$, are given by 
\[\zeta_{i,j}=\begin{cases}
(8(j-1)+i)\bmod{64} & \text{if~} j \text{~is odd,}\\
(8j-(i-1))\bmod{64} & \text{if~}j \text{~is even.}\\
\end{cases}\] 
   \begin{figure}[htbp]
  \centering
   \begin{tikzpicture}[scale=1]
   \node at (-4.5,0){$\mathbb{M}=$};
   \node at (0,0){\addtolength{\tabcolsep}{2pt}
\renewcommand{\arraystretch}{1.2}
 \centering
\begin{tabular}{|m{0.4cm}|m{0.4cm}|m{0.4cm}|m{0.4cm}|m{0.4cm}|m{0.4cm}|m{0.4cm}|m{0.4cm}|}
\hline
1 & 16 & 17 & 32 & 33 & 48 & 49 & 64 \\[2pt] \hline
2 & 15 & 18 & 31 & 34 & 47 & 50 & 63 \\ \hline
3 & 14 & 19 & 30 & 35 & 46 & 51 & 62 \\ \hline
4 & 13 & 20 & 29 & 36 & 45 & 52 & 61 \\ \hline
5 & 12 & 21 & 28 & 37 & 44 & 53 & 60 \\ \hline
6 & 11 & 22 & 27 & 38 & 43 & 54 & 59 \\ \hline
7 & 10 & 23 & 26 & 39 & 42 & 55 & 58 \\ \hline
8 & 9  & 24 & 25 & 40 & 41 & 56 & 57 \\ \hline
\end{tabular}};
\node at (4.5,0){\addtolength{\tabcolsep}{2pt}
\renewcommand{\arraystretch}{1.2}
 \centering
\begin{tabular}{|m{0.4cm}|} \hline
4 \\ \hline  4 \\\hline  4 \\\hline 4 \\\hline 4 \\ \hline 4 \\\hline 4 \\ \hline  4\\
 \hline
\end{tabular}};
\node at (0,-2.8)
{\addtolength{\tabcolsep}{2pt}
\renewcommand{\arraystretch}{1.2}
 \centering
\begin{tabular}{|m{0.4cm}|m{0.4cm}|m{0.4cm}|m{0.4cm}|m{0.4cm}|m{0.4cm}|m{0.4cm}|m{0.4cm}|}
\hline
36 & 36  & 36  & 36  & 36  & 36  & 36  & 36 \\ \hline
\end{tabular}
};

\end{tikzpicture}
\caption{A construction of the array $\mathbb{M}$}
\end{figure}
\noindent At this stage, each row-sums of $\mathbb{M}$ is  congruent to $4$ modulo $64$, while each column-sums of $\mathbb{M}$ is congruent to $36$ modulo $64$. Therefore, a further rearrangement of the entries is required.\\
 
\noindent \textbf{Step 2. Construction of the row subgroup.}  
In this step, we interchange certain entries of the array to adjust the row-sums. 
Specifically, we perform the following swaps:  
\[
\zeta_{1,1}\rightleftharpoons \zeta_{3,1}, \quad 
\zeta_{2,2}\rightleftharpoons \zeta_{4,8}, \quad 
\zeta_{5,1}\rightleftharpoons \zeta_{7,7}, \quad 
\zeta_{6,2}\rightleftharpoons \zeta_{8,8}.
\]
\vspace{-0.5cm}
\begin{figure}[h]
   \centering
   \begin{tikzpicture}[scale=1]
    \node at (-4.5,0){$\mathbb{M}'=$};
   \node at (0,0){\addtolength{\tabcolsep}{2pt}
\renewcommand{\arraystretch}{1.2}
 \centering
\begin{tabular}{|c|c|c|c|c|c|c|c|}
\hline
\textcolor{blue}{\textbf{61}} & 16 & 17 & 32 & 33 & 48 & 49 & 64 \\ \hline
2  & \textcolor{blue}{\textbf{51}} & 18 & 31 & 34 & 47 & 50 & 63 \\ \hline
3  & 14 & 19 & 30 & 35 & 46 & \textcolor{blue}{\textbf{15}}& 62 \\ \hline
4  & 13 & 20 & 29 & 36 & 45 & 52 & \textcolor{blue}{\textbf{1}}  \\ \hline
\textcolor{blue}{\textbf{57}} & 12 & 21 & 28 & 37 & 44 & 53 & 60 \\ \hline
6  & \textcolor{blue}{\textbf{55}}& 22 & 27 & 38 & 43 & 54 & 59 \\ \hline
7  & 10 & 23 & 26 & 39 & 42 & \textcolor{blue}{\textbf{11}} & 58 \\ \hline
8  & 9  & 24 & 25 & 40 & 41 & 56 & \textcolor{blue}{\textbf{5}} \\ \hline
\end{tabular}};
\node at (4.6,0){\addtolength{\tabcolsep}{2pt}
\renewcommand{\arraystretch}{1.2}
 \centering
\begin{tabular}{|m{0.4cm}|}
\hline 0 \\ \hline 40 \\ \hline  32 \\ \hline 8 \\ \hline 56 \\ \hline  48 \\ \hline 24 \\  \hline 16\\
\hline
\end{tabular}};
% \draw [decorate,decoration={brace,amplitude=10pt}, black](4.9,2.2)-- (4.9,-2.2);
\node at (4.6, 2.8){$\langle8\rangle$};

\node at (0,-2.8){\addtolength{\tabcolsep}{2pt}
\renewcommand{\arraystretch}{1.2}
 \centering

\begin{tabular}{|m{0.4cm}|m{0.4cm}|m{0.4cm}|m{0.4cm}|m{0.4cm}|m{0.4cm}|m{0.4cm}|m{0.4cm}|}
\hline
20  & 52  & 36  & 36  & 36  & 36  & 20  & 52 \\
\hline
\end{tabular}
};
\end{tikzpicture}
\caption{The set of all row-sums forms the subgroup $\langle 8 \rangle$ in $\mathbb{M}'$ array}
\end{figure}
%-------------------
\\
At this stage, the set of all row-sums forms the subgroup $\langle 8 \rangle$ of $\mathbb{Z}_{64}$. That is $\mathbb{M}'$ is $\langle 8 \rangle $ kotzig array.  
However, the column sums do not yet form a subgroup, and therefore an additional rearrangement is required.
 
\noindent \textbf{Step 3. Construction of the column subgroup.}  
In this step, we interchange selected entries of the array in order to adjust the column-sums.  
The following swaps are carried out:  
\[
\zeta_{1,1} \rightleftharpoons \zeta_{1,5}, \quad 
\zeta_{2,2} \rightleftharpoons \zeta_{2,6}, \quad 
\zeta_{3,1}\rightleftharpoons \zeta_{3,7}, \quad 
\zeta_{4,4}\rightleftharpoons \zeta_{4,8}.
\]
\vspace{-0.5cm}
\begin{figure}[h]
   \centering
   \begin{tikzpicture}[scale=1]
   \node at (-4.5,0){$\mathbb{M}''=$};
   \node at (0,0){\addtolength{\tabcolsep}{2pt}
\renewcommand{\arraystretch}{1.2}
 \centering 
\begin{tabular}{|m{0.4cm}|m{0.4cm}|m{0.4cm}|m{0.4cm}|m{0.4cm}|m{0.4cm}|m{0.4cm}|m{0.4cm}|}
\hline
\textcolor{red}{\textbf{33}} & 16 & 17 & 32 & \textcolor{red}{\textbf{61}} & 48 & 49 & 64 \\ \hline
2  & \textcolor{red}{\textbf{47}} & 18 & 31 & 34 & \textcolor{red}{\textbf{51}} & 50 & 63 \\ \hline
3  & 14 & \textcolor{red}{\textbf{15}} & 30 & 35 & 46 & \textcolor{red}{\textbf{19}} & 62 \\ \hline
4  & 13 & 20 & \textcolor{red}{\textbf{1}} & 36 & 45 & 52 &   \textcolor{red}{\textbf{29}}  \\ \hline
57 & 12 & 21 & 28 & 37 & 44 & 53 & 60 \\ \hline
6  & 55 & 22 & 27 & 38 & 43 & 54 & 59 \\ \hline
7  & 10 & 23 & 26 & 39 & 42 & 11 & 58 \\ \hline
8  & 9  & 24 & 25 & 40 & 41 & 56 & 5  \\ \hline
\end{tabular}};
\node at (4.6,0){\addtolength{\tabcolsep}{2pt}
\renewcommand{\arraystretch}{1.2}
 \centering
\begin{tabular}{|m{0.4cm}|}
\hline 0 \\ \hline 40 \\ \hline  32 \\ \hline 8 \\ \hline 56 \\ \hline  48 \\ \hline 24 \\  \hline 16\\
\hline
\end{tabular}};
% \draw [decorate,decoration={brace,amplitude=10pt}, black](4.9,2.2)-- (4.9,-2.2);
\node at (4.6, 2.8){$\langle8\rangle$};
\node at (0,-2.8){
\addtolength{\tabcolsep}{2pt}
\renewcommand{\arraystretch}{1.2}
 \centering

\begin{tabular}{|m{0.4cm}|m{0.4cm}|m{0.4cm}|m{0.4cm}|m{0.4cm}|m{0.4cm}|m{0.4cm}|m{0.4cm}|}
\hline
56  & 48  & 32  & 8   & 0   & 40  & 24  & 16\\
\hline
\end{tabular}
};
% \draw [decorate,decoration={brace,amplitude=10pt}, black](3,-3.2)-- (-3,-3.2);
\node at (-4.3, -2.8){$\langle8\rangle$};
\end{tikzpicture}
\caption{The set of all column-sums forms the subgroup $\langle 8 \rangle$ in $\mathbb{M}''$ array}
\end{figure}

\noindent At this step, the set of all column sums forms a subgroup of $\mathbb{Z}_{64}$ again. Therefore, the array $\mathbb{M}''$ is a subgroup magic rectangle with row-subgroup and column-subgroup are $\langle 8 \rangle$ in $\mathbb{Z}_{64}$. 
\par In the next theorem, we recursively construct higher-order $m \times n$ subgroup magic rectangles by expanding the above construction.
\begin{lemma}\label{karlem}
   If $m\equiv 0\bmod{4}$ and $n=km,$ where $k$ is odd, then there exists an $m\times n$ subgroup magic rectangle over $\mathbb{Z}_{mn}$.
\end{lemma}
\begin{proof}
    The construction is given in three steps. 
    \\
    \noindent Step 1. Without loss of generality, assume $m \leq n$. By Theorem~\ref{thmkar},  let $\mathbb{M}$ be the $m \times m$ subgroup magic rectangle.\\
    \noindent Step 2. Let $\mathbb{A}$ be an $m \times n$ matrix with $n = km$, and $k$ is odd. 
Then $\mathbb{A}$ can be expressed as a concatenation of $k$ block matrices, each of size $m \times m$; that is,
\[
\mathbb{A} = \begin{bmatrix} 
\mathbb{A}_1 ~~& \mathbb{A}_2 ~~& \cdots & ~~\mathbb{A}_k
\end{bmatrix}
\]
Now, for each $i\in [1,k]$, we define the $m\times m$ matrix $\mathbb{A}_i$ such that \[\mathbb{A}_i=\begin{cases}
    \mathbb{M} & \text{if} ~i=1, \\
    m^2\mathbb{I}+ \mathbb{A}_{i-1} & \text{if}~ i \in [2,k].\end{cases},\]
where $\mathbb{I}$ is the $m\times m$ identity matrix.\\
\noindent Step 3. In this step, we prove that $\mathbb{A}$ forms a subgroup magic rectangle. Since $\mathbb{M}$ is a subgroup magic rectangle over $\mathbb{Z}_{m^2}$, the row-subgroup of $\mathbb{M}$ over $\mathbb{Z}_{m^2}$ is $\langle m\rangle$. We consider the expanded matrix $\mathbb{A}$ over $\mathbb{Z}_{k m^2}$, where $k$ is an odd integer and $m \equiv 0 \pmod{4}$, obtained by replicating and shifting $\mathbb{M}$ in blocks.

\noindent We claim that the following subset $H_1$ of the group $\mathbb{Z}_{k m^2}$,
\[
H_1 = \frac{(k-1
)k}{2}m^3 \bmod{km^2}+  \langle m \rangle   
\]
forms a subgroup under addition modulo $k m^2$.
Thus, \begin{align*}
    H_1 & = \{ikm\bmod{m^2}+ \frac{(k-1)k}{2}m^3\bmod{km^2}: i \in [1, m]\}\\
    & = \{ikm\bmod{m^2}:  i \in [1, m]\} = \{ikm\bmod{km^2}:  i \in [1, m]\}.
\end{align*}
\noindent Therefore, $H_1$ is  subgroup of order $m$. Similarly, the column sums of $\mathbb{A}$, are in the set $H_2$. Then, \begin{align*}
    H_2&= \{ilm\bmod{m^2}+(i-1)m^3\bmod{km^2} : i \in[1,k], l \in[1,m]\}
\end{align*} 
We first show that every element of \(\mathcal H_2\) belongs to the subgroup
\(\langle m\rangle\) of \(\mathbb Z_{km^2}\). Let
\[
h=ilm \bmod{m^2}+(i-1)m^3 \bmod{km^2}
\in \mathcal H_2.
\]
Since
\(
ilm \pmod{m^2}=m(il \pmod m),
\)
there exists an integer \(r\in[0,m-1]\) such that
\(
ilm \pmod{m^2}=rm.
\)
Hence
\(
h=rm+(i-1)m^3=m\bigl(r+(i-1)m^2\bigr).
\)
Therefore \(h\) is a multiple of \(m\), and thus
\(
h\in \langle m\rangle.
\)
Since \(h\) was arbitrary,
\(
\mathcal H_2\subseteq \langle m\rangle\). Now consider the subgroup generated by $\langle m\rangle$ in $\mathbb{Z}_{km^2}$. Since,
\[ o(m)=\frac{km^2}{\gcd(km^2,m)}=
\frac{km^2}{m}=
km.
\]
Thus \(\langle m\rangle\) is a subgroup of \(\mathbb Z_{km^2}\) of order \(km\). Observe also that
\(
m^3=m^2\cdot m,
\)
and hence
\(
m^3\in\langle m\rangle.
\)
Consequently,
\(
\langle m,m^3\rangle=\langle m\rangle.
\) Since the column sums of \(\mathbb A\) form a subgroup of order \(km\), we have
\(
|\mathcal H_2|=km.
\)
Combining this with
\(
\mathcal H_2\subseteq \langle m\rangle
~\text{and}~
|\langle m\rangle|=km,
\)
it follows that
\(
\mathcal H_2=\langle m\rangle.
\)
Therefore, 
\[H_2=\langle m\rangle=
\langle m,m^3\rangle=
\{tm \bmod{km^2}: t\in\mathbb Z\},
\]
\noindent which is a subgroup of \(\mathbb Z_{km^2}\) of order \(km\). This completes the proof.\end{proof}
\begin{example}
    Consider Example \(\ref{ex}\). Fix $k=3$. 
\end{example}

\begin{figure}[h]
   \begin{tikzpicture}[scale=0.5]
   \node at (-5,0){\resizebox{1\textwidth}{!}{\small\addtolength{\tabcolsep}{0pt}
\renewcommand{\arraystretch}{1}
\begin{tabular}{|l|l|l|l|l|l|l|l||l|l|l|l|l|l|l|l||l|l|l|l|l|l|l|l|}
\hline
33 & 16 & 17 & 32 & 61 & 48 & 49 & 64 &97&80&81&96&125&112&113&128 &161&144&145&160&189&176&177&192\\ \hline
2  & 47 & 18 & 31 & 34 & 51 & 50 & 63 &66&111&82&95&98&115&114&127 &130&175&146&159&162&179&178&191 \\ \hline
3  & 14 &15 & 30 & 35 & 46 &19 & 62 &67&78&79&94&99&110&83&126 &131&142&143&158&163&174&147&190 \\ \hline
4  & 13 & 20 & 1 & 36 & 45 & 52 &  29 &68&77&84&65&100&109&116&93 &132&141&148&129&164&173&180&157 \\ \hline
57 & 12 & 21 & 28 & 37 & 44 & 53 & 60 &121&76&85&92&101&108&117&124 &185&140&149&156&165&172&181&188 \\ \hline
6  & 55 & 22 & 27 & 38 & 43 & 54 & 59 &70&119&86&91&102&107&118&123&134&183&150&155&166&171&182&187 \\ \hline
7  & 10 & 23 & 26 & 39 & 42 & 11 & 58&71&74&87&90&103&106&75&122 &135&138&151&154&167&170&139&176 \\ \hline
8  & 9  & 24 & 25 & 40 & 41 & 56 & 5 &72&73&88&89&104&105&120&69 &136&137&152&153&168&169&184&133 \\ \hline
\end{tabular}}};
\node at (-5,-3.8){\resizebox{1\textwidth}{!}{\small\addtolength{\tabcolsep}{0pt}
 \renewcommand{\arraystretch}{1}
  \centering
\begin{tabular}{|l|l|l|l|l|l|l|l||l|l|l|l|l|l|l|l||l|l|l|l|l|l|l|l|}
\hline
120 &	176&	160	 & 8&	128 &	168	& 152&	16&	56&	112&	96&	136&	64&	104&	88&	144&	184 &	48	& 32& 	72	& 0& 	40&	24 & 	80 \\ \hline
\end{tabular}}};
\node at (12.5, 3.8){\small $H_1$};
\node at (12.5, -3.8){\small $H_2$};
\node at (12.5,0){\resizebox{0.056\textwidth}{!}{\small\addtolength{\tabcolsep}{0pt}
 \renewcommand{\arraystretch}{1}
  \centering
\begin{tabular}{|l|}
\hline
0 \\ \hline
120\\ \hline
96\\ \hline
24\\ \hline
168\\ \hline
144\\ \hline
72\\ \hline
48\\ \hline
\end{tabular}}};
\draw [decorate,decoration={brace,amplitude=10pt}, black](0,3)-- (11,3);
\draw [decorate,decoration={brace,amplitude=10pt}, black](-11.2,3)-- (-0.3,3);
\draw [decorate,decoration={brace,amplitude=10pt}, black](-21,3)-- (-12,3);
 \node at (5.5, 4.5){\small$\mathbb{A}_3=64\mathbb{I}+\mathbb{A}_2$};
 \node at (-5.5,4.5){\small$\mathbb{A}_2=64\mathbb{I}+\mathbb{A}_1$};
  \node at (-16.5, 4.5){$\mathbb{A}_1$};
\end{tikzpicture}
 \caption{$8 \times 24$ subgroup magic rectangle over $\mathbb{Z}_{192}$ }
\end{figure}
\begin{theorem}
     If $m\equiv 0\bmod{4}$ and $n=km,$, $k$ is odd, then there exists a $(km^2, m, km)$ modular balanced partition design.
\end{theorem}
\begin{proof}
    The proof follows from Lemma \ref{lemmamain} and Lemma \ref{karlem}. 
\end{proof}
 \begin{theorem}
Let $X=\mathbb{Z}_\nu$, where   $\nu=\kappa\lambda$ and $\gcd(\kappa,\lambda)=1$. Suppose $\varphi:\mathbb{Z}_\kappa\times\mathbb{Z}_\lambda\to\mathbb{Z}_\nu$ is bijective map given by \(\varphi(s,t)\equiv (s\lambda+t\kappa) \bmod{\nu}.
\)
Let $A_i= \{\varphi(i,t): t\in\mathbb{Z}_\lambda\}$ and $B_j = \{\varphi(s,j): s\in\mathbb{Z}_\kappa \}$, where $i\in[0,\kappa-1]$ and $j\in[0,\lambda-1]$. Then $(X,\mathcal{A},\mathcal{B})$ 
is a modular balanced partition design, where $\mathcal{A}=\{A_i: i\in [0, \kappa-1]\}$, and $\mathcal{B}=\{B_j: j\in [0,\lambda-1]\}$.
\end{theorem}

\begin{proof}
Since $\gcd(\kappa,\lambda)=1$, we see that $\langle \lambda\rangle = \langle \kappa \rangle = \mathbb{Z}_\nu$.
Now, any $x\in\mathbb{Z}_\nu$ is uniquely represented  as $x\equiv (s\lambda+t\kappa)\bmod{\nu}$, where $s\in\mathbb{Z}_\kappa$, $t\in\mathbb{Z}_\lambda$. Thus,  $\varphi$ is a bijection.
For fixed $i$, $A_i=\{\varphi(i,t):t\in\mathbb{Z}_\lambda\}$ contains exactly $\lambda$ elements,
one for each $t$, and distinctness follows from bijectivity of $\varphi$ in the second coordinate.
Similarly, for fixed $j$,  $B_j=\{\varphi(s,j):s\in\mathbb{Z}_\kappa\}$ has exactly $\kappa$ elements.
As $i$ (resp.\ $j$) varies,  $A_i$'s (resp.\ $B_j$'s) are disjoint and their respective union is $X$.
Given $i\in\mathbb{Z}_\kappa$ and $j\in\mathbb{Z}_\lambda$, we have
\[
A_i\cap B_j=\{\varphi(i,j)\},
\]
since any element common to both must have the same coordinates $(i,j)$ under the bijection $\varphi$.
Thus $|A_i\cap B_j|=1$.  Now we proceed to verify the modular sum condition for rows.
\[
\sum_{a\in A_i}a \;\equiv\; \sum_{t=0}^{\lambda-1} \bigl(i\lambda+t\kappa\bigr)
\;=\;\lambda (i\lambda)\;+\;\kappa\sum_{t=0}^{\lambda-1} t
\;=\;\lambda^2 i\;+\;\kappa\frac{\lambda(\lambda-1)}{2}\equiv\; \lambda^2 i \;\equiv\; i\lambda \bmod{\kappa\lambda}.
\]

\noindent Similarly, the modular sum condition for columns are given by, 
 
\[
\sum_{b\in B_j}b \;\equiv\; \sum_{s=0}^{\kappa-1} \bigl(s\lambda+j\kappa\bigr)
\;=\;\lambda\sum_{s=0}^{\kappa-1} s\;+\;\kappa( j\kappa)
\;=\;\lambda\frac{\kappa(\kappa-1)}{2}\;+\;\kappa^2 j\;\equiv\; \kappa^2 j \;\equiv\; j\kappa \bmod{\kappa\lambda}.
\] This completes the proof.
\end{proof}

 \section{Isomorphism and automorphism of \MBPD}
 In this section, we introduce and study the concept of isomorphism between two modular balanced partition designs. 
 \begin{definition} Suppose $(X, \mathcal{A}, \mathcal{B})$ and $(\mathcal{Y}, \mathcal{C}, \mathcal{D})$ are two modular balanced partition designs with $|X|=|\mathcal{Y}|$.  $(X, \mathcal{A}, \mathcal{B})$ and $(\mathcal{Y}, \mathcal{C}, \mathcal{D})$  are isomorphic if there exists a bijection $\psi: X\rightarrow \mathcal{Y}$ such that \[[\{\psi(x): x \in A\}: A \in \mathcal{A}]= \mathcal{C} \textnormal{~and~}[\{\psi(x'): x'\in B\}: B \in \mathcal{B}]= \mathcal{D}.\]
 \noindent In other words, if we replace  every point $x\in X$ by its image $\psi(x)$, then the collection of blocks in $\mathcal{A}$ and $\mathcal{B}$ is transformed into $\mathcal{C}$ and $\mathcal{D}$, respectively. Such a bijection $\psi$ is called an isomorphism.
\end{definition}

Suppose that \((X, \mathcal{A}, \mathcal{B})\) is a modular balanced partition design. An \textit{automorphism} of \((X, \mathcal{A}, \mathcal{B})\) is defined as an isomorphism of this design with itself.  In this context, the bijection \(\psi :X \rightarrow \mathcal{Y}\)  such that: 
\[
[ \{\psi(x) : x \in A\} : A \in \mathcal{A} ]= \mathcal{A} 
\quad \text{and} \quad 
[\{\psi(x) : x \in B\} : B \in \mathcal{B}] = \mathcal{B}.
\]
The identity map on $X$ is always a trivial automorphism, but for $\MBPD$ we have additional non-trivial automorphisms. 
\begin{example}\label{ex} Consider $(9, 3,3)-\MBPD$ design with $X  =\mathbb{Z}_{9}, 
   \mathcal{A}  =\{\{0,3,6\}, \{1,4,7\},\{2,5,8\}\}$ and $
    \mathcal{B}  = \{\{0,1,8\}, \{3,4,5\},\{6,7,2\} \}$.
We define a bijection $\psi$ from $\mathbb{Z}_9$ to itself  by $\psi(x)=(x+1)\bmod{9}$ on $(9, 3,3)-\MBPD$. Then the blocks in $\mathcal{A}$ and $\mathcal{B}$ become, 
\begin{align*}
 X&  =\mathbb{Z}_{9},\\ 
    \mathcal{A} &  =\{\{1,4,7\}, \{2,5,8\},\{3,6,0\}\}\\
    \mathcal{B} &  = \{\{1,2,0\}, \{4,5,6\},\{7,8,3\} \}.
\end{align*}
Thus, \(\psi\) is an isomorphism of two \(\MBPD\)s as well as an automorphism of \(\MBPD\)s. 
\end{example}

\begin{note}
    In Example \(\ref{ex}\), the map $\psi$ from  \((X, \mathcal{A}, \mathcal{B})\) to \((X, \mathcal{C}, \mathcal{D})\) defined by $\psi(x)=(x+3)\bmod{9}$ is not an isomorphism.
\end{note}
\begin{note}It is often convenient to represent a bijection
\(\psi:X\to X\) by its disjoint cycle decomposition.
For each \(x\in X\), consider the sequence
\(
x,\ \psi(x),\ (\psi\circ\psi)(x),\ (\psi\circ\psi\circ\psi)(x),\ \ldots.
\)
Since \(X\) is finite and \(\psi\) is a bijection, there exists a
smallest positive integer \(m\) such that
\(
(\underbrace{\psi\circ\psi\circ\cdots\circ\psi}_{m\text{ times}})(x)=x.
\)
This gives rise to the cycle
\(
\bigl(x\ \psi(x)\ (\psi\circ\psi)(x)\ \cdots\
(\underbrace{\psi\circ\cdots\circ\psi}_{m-1\text{ times}})(x)\bigr).
\)
Repeating this process for elements not yet contained in any cycle yields a
collection of pairwise disjoint cycles whose union is \(X\). The
lengths of these cycles sum to \(|X|=\nu\). \end{note}

\begin{theorem}
    The collection of all automorphisms of a modular balanced partition design ($X,\mathcal{A},\mathcal{B})$ forms a group under the operation of function composition.
\end{theorem}
\begin{proof}
Let $(X, \mathcal{A}, \mathcal{B})$ be a modular balanced partition design. Suppose $
\operatorname{Aut}(X,\mathcal{A},\mathcal{B})$ is the collection of all bijection 
$\psi:X \to X$ such that $\{\psi(x):x\in A\}:A\in\mathcal{A}]=\mathcal{A}$ and $
[\{\psi(x):x\in B\}:B\in\mathcal{B}]=\mathcal{B}\}$. Clearly, $\operatorname{Aut}(X,\mathcal{A},\mathcal{B})\ne \emptyset$, since the identity function $\mathrm{id}_{X}$ with $\mathrm{id}_{X}(x)=x,\forall x\in X$ is always in. Suppose $\psi, \varphi \in \operatorname{Aut}(X, \mathcal{A}, \mathcal{B})$. 
For any $A \in \mathcal{A}$, we have $\{\varphi(x) : x \in A\}=A'$ and $A'\in \mathcal{A}$. Now,  
\[
 \{(\psi\circ \varphi)(x) : x \in A\}=\{\psi(\varphi(x)) : x \in A\} = \{\psi(y) : y \in \{\varphi(x) : x \in A\}=\{\psi(y) : y \in A'\}.
\]
Since $\psi$ maps every block of $\mathcal{A}$ onto another block of $\mathcal{A}$, we have $\{\psi(y) : y \in A'\}=A''$, where $A''\in \mathcal{A}$ and thus, $\psi \circ \varphi$ maps each block of $\mathcal{A}$ to a block of $\mathcal{A}$. 
A similar argument applies to $\mathcal{B}$ and thus $\psi \circ \varphi \in \operatorname{Aut}(X, \mathcal{A}, \mathcal{B})$.  
Since the composition of mappings is associative, we have
\[
(\phi \circ \psi) \circ \varphi = \phi \circ (\psi \circ \varphi), \textnormal{~for~all~} \rho, \psi, \varphi \in \operatorname{Aut}(X, \mathcal{A}, \mathcal{B}).
\] Suppose $\psi \in \operatorname{Aut}(X, \mathcal{A}, \mathcal{B})$. For any $A \in \mathcal{A}$,
 and $B \in \mathcal{B}$,
 we have
$\{\psi(\mathrm{id}_{X}(x)) : x \in A\} = \{\psi(x) : x \in A\} = \{\mathrm{id}_{X}(\psi(x)) : x \in A\}$  and  $
\{\psi(\mathrm{id}_{X}(x)) : x \in B\} = \{\psi(x) : x \in B\} = \{\mathrm{id}_{X}(\psi(x)) : x \in B\}$ and hence the identity map $\mathrm{id}_{X}$ serves as the identity element. Moreover, if $\psi \in \operatorname{Aut}(X, \mathcal{A}, \mathcal{B})$, then $\psi$ is a bijection. 
For each $A \in \mathcal{A}$, let $A' =\psi(A)= \{\psi(x):x \in A\}$ and $A'\in \mathcal{A}$. Now, $\psi^{-1}(A') = \psi^{-1}(\psi(A)) = (\psi^{-1}\circ \psi)(A) = \mathrm{id}_X(A)  = A$. 
Thus, $\psi^{-1}$ maps $\mathcal{A}$ onto itself, and  similar argument applies to $\mathcal{B}$. 
Therefore, $\psi^{-1} \in \operatorname{Aut}(X, \mathcal{A}, \mathcal{B})$. In addition, for any $A\in \mathcal{A}$, and $B\in \mathcal{B}$, we have 
 $(\psi\circ \psi^{-1})(A)=\mathrm{id}_X(A)=(\psi^{-1}\circ \psi)(A)$ and $\psi\circ \psi^{-1}(B)=\mathrm{id}_X(B)=\psi^{-1}\circ \psi(B)$.
\noindent Therefore, $\operatorname{Aut}(X, \mathcal{A}, \mathcal{B})$ is a group under composition. 
\end{proof}
  
\begin{corollary}{\label{cor3}}
   If $\nu$ is the cardinality of the set $X$, then $\textnormal{Aut}(X,\mathcal{A}, \mathcal{B})$ is a subgroup of the symmetric group $S_{\nu}$.
\end{corollary}
%Since the subgroups of a symmetric group are permutation groups, by Corollary \ref{cor3}, $\textnormal{Aut}(X,\mathcal{A}, \mathcal{B})$ are examples of the permutation group.
\begin{example}
    The $(9,3,3)-\MBPD$ in Example \textnormal{\ref{ex}},
has another isomorphism, $\phi=(2~~ 3 ~~8~~4~~5\\~~6~~0~~1~~7)$. The composition $\varphi=\psi \circ \phi$ is defined as $\varphi(x)=\phi(\psi(x))$ for all $x\in X$. Hence, we have, 
$\psi=(1~~2~~3~~4~~5~~6~~7~~8~~0)$ and therefore $\varphi=(3~~4~~0~~5~~6~~7~~1~~2~~8)$. Consequently, 
\begin{align*}
 X&  =\mathbb{Z}_{9},\\ 
    \mathcal{A} &  =\{\{3,5,1\}, \{4,6,2\},\{8,7,0\}\}\\
    \mathcal{B} &  = \{\{3,4,8\}, \{5,6,7\},\{1,2,0\} \}.
\end{align*}
\noindent Thus,
$\varphi$ is an isomorphism of $\MBPD$. \end{example}

\begin{lemma}
\label{lem:aut-action-blocks}
Let $\Gamma^*=\operatorname{Aut}(X,\mathcal{A},\mathcal{B})$. Then $\Gamma^*$ acts naturally on the set of $\mathcal{A}$-blocks and on the set of $\mathcal{B}$-blocks. Consequently $\Gamma^*$ embeds as a subgroup of $\operatorname{Sym}(\mathcal{A})\times\operatorname{Sym}(\mathcal{B})$.
\end{lemma}

\begin{proof}
If $\psi\in \Gamma^*$ and $A\in\mathcal{A}$ then $\psi(A)=\{\psi(x):x\in A\}$ is again an $\mathcal{A}$-block by definition of $\Gamma^*$, so $\psi$ permutes $\mathcal{A}$. The map $\Gamma^*\to\operatorname{Sym}(\mathcal{A})\times\operatorname{Sym}(\mathcal{B})$ sending $\psi$ to $(\psi|_{\mathcal{A}},\psi|_{\mathcal{B}})$ is a group homomorphism and injective because $\psi$ is determined by its action on points; hence the embedding.
\end{proof}

\begin{theorem}
\label{thm:development}
Let $(X, \mathcal{A}, \mathcal{B})$ be a modular balanced partition design. 
Suppose there exists a cyclic subgroup 
$\mathcal{C}=\langle \tau \rangle \leq \operatorname{Aut}(X, \mathcal{A}, \mathcal{B})$
acting {regularly} on $X$, that is, $|\mathcal{C}|=\nu=|X|$, and for every pair of points 
$x, y \in X$ there exists a unique $\tau^i \in \mathcal{C}$ such that $\tau^i(x)=y$. 
Let $\mathcal{R} \subseteq \mathcal{A}$ be a set of representatives of the $\mathcal{C}$-orbits on $\mathcal{A}$. 
Then $
\mathcal{A} = \big\{\, \tau^i(R) : R \in \mathcal{R},\, i \in [0,  \nu-1] \,\big\},
$
and similarly,
$
\mathcal{B} = \big\{\, \tau^i(S) : S \in \mathcal{S},\,i \in [0,  \nu-1] \big\},
$, 
where $\mathcal{S}$ is a set of orbit representatives of $\mathcal{C}$ on $\mathcal{B}$. 
In particular, $(X, \mathcal{A}, \mathcal{B})$ can be completely generated by 
developing a small collection of base blocks under the cyclic action of $\mathcal{C}$.
\end{theorem}

\begin{proof}
Since $\mathcal{C}$ acts regularly on $X$, it acts \emph{sharply transitively} on the point set, 
meaning that each point of $X$ can be obtained from any other by the unique action 
of an element of $\mathcal{C}$. Consequently, the group action induces a partition of the block sets 
$\mathcal{A}$ and $\mathcal{B}$ into disjoint orbits under $\mathcal{C}$.

Let $\mathcal{R}$ denote a complete set of orbit representatives from $\mathcal{A}$.
For each representative block $R \in \mathcal{R}$, the orbit of $R$ under $\mathcal{C}$ is given by 
\[
\operatorname{Orb}_C(R) = \{\, \tau^i(R) : i \in [0,  \nu-1] \}.
\]
By regularity, each orbit has exactly $\nu$ distinct images, and every block of $\mathcal{A}$ 
appears in one of these orbits. Hence,
\[
\mathcal{A} = \bigcup_{R \in \mathcal{R}} \operatorname{Orb}_C(R)
= \big\{\, \tau^i(R) : R \in \mathcal{R},\, i \in [0,  \nu-1] \big\}.
\]

\

The same reasoning applies to $\mathcal{B}$: selecting a set $\mathcal{S}$ of orbit representatives 
and applying the cyclic action of $\mathcal{C}$ yields all blocks of $\mathcal{B}$. 

Thus, the entire modular balanced partition design is obtained by the cyclic development of 
a smaller subset of base blocks under the action of a regular automorphism group $\mathcal{C}$.
\end{proof}

\subsection{Construction of~\MBPD~using affine automorphism}

In the section, we assume $X=\mathbb{Z}_{\nu}$. By using affine automorphisms on finite cyclic groups $\mathbb{Z}_{\nu}$, we construct a non-isomorphic  modular balanced partition designs.

\begin{definition}[\cite{cameron1999permutation}]Let $\mathbb{Z}_{\nu}^{\times}$ denote the group of units of $\mathbb{Z}_{\nu}$. Define
\(
\operatorname{AG}(\nu)=\{\tau_{c,d}: c\in \mathbb{Z}_{\nu}^{\times},\ d\in \mathbb{Z}_{\nu}\},
\)
where each permutation $\tau_{c,d}$ of $\mathbb{Z}_{\nu}$ is given by
\(
\tau_{c,d}(x)=cx+d \pmod{\nu}.
\)
Then $\AG(\nu)$ forms a subgroup of the symmetric group on $\mathbb{Z}_{\nu}$ under composition. This group is called the \emph{affine group} (or the group of \emph{affine automorphisms}) of $\mathbb{Z}_{\nu}$.
\end{definition}
\begin{note}
     The order of the group $\AG(\nu)$ is $\nu\varphi(\nu),$ where $\varphi$ is Euler’s totient function.
\end{note}

   \begin{theorem}{\label{thm5.10}}
    Let $X = \mathbb{Z}_\nu$ and $(X, \mathcal{A}, \mathcal{B})$ be a modular balanced partition design. 
    If $\tau_{c,d} \in \AG(\nu)$ is an affine automorphism on $X$, then 
   $
    \big(X, \, \tau_{c,d}(\mathcal{A}), \, \tau_{c,d}(\mathcal{B}) \big)
    $
    is again a modular balanced partition design.
\end{theorem}

\begin{proof}
Since $\tau_{c,d}(X) = X$, it suffices to prove that the three conditions of Definition~\ref{def1} are satisfied by 
\(
(X, \tau_{c,d}(\mathcal{A}), \tau_{c,d}(\mathcal{B})).
\)

\noindent (i) Partition and modular sums for $\mathcal{A}$. 
As $\tau_{c,d}$ is a bijection on $X$, the sets $\tau_{c,d}(A_1),\dots,\tau_{c,d}(A_\kappa)$ form a partition of $X$. Since each block $A_i$ has cardinality $\lambda$, $|\tau_{c,d}(A_i)|=\lambda$.   We have,
\begin{align}{\label{eqn10}}
\sum_{a\in A_i}a \equiv i\lambda \bmod{\nu}, ~~~& i\in [1,\kappa].
\end{align}
Now, for each $i\in [1,\kappa]$, $
\tau_{c,d}(\sum_{a\in A_i}a) = \sum_{a\in A_i}(ca+d)\bmod{\nu} \equiv \bigl(c\,\sum_{a\in A_i}a + \lambda d \bigr)\bmod{\nu}$ and hence 
\[
\tau_{c,d}(\sum_{a\in A_i}a) \equiv \left(ci\lambda + \lambda d\right) \bmod{\nu}.
\]
Since $\gcd(c,\nu)=1$, the map $i \mapsto ci \bmod{\nu}$ is a permutation of $[1,\kappa]$. 
Consequently the set
\(
\{ ic\lambda\bmod{\nu} : i\in [1,\kappa] \} 
\)
is simply a reordering of $\{ i\lambda\bmod{\nu} : i\in [1,\kappa] \}$. 
Adding the fixed shift $\lambda d$ to each element preserves this property, so overall we obtain
\[
\{ \tau_{c,d}(\sum_{a\in A_i}a) : i\in [1,\kappa] \}
= \{ j\lambda \bmod{\nu} : j\in [1,\kappa] \}.
\]
\noindent (ii) Partition and modular sums for $\mathcal{B}$. 
An analogous argument applies to $\mathcal{B}=\{B_j: j \in [1,\lambda]\}$. The family $\tau_{c,d}(\mathcal{B})$ partitions $X$ into $\lambda$ blocks of size $\kappa$, and their block sums are permuted within the set $\{ j\kappa \bmod{\nu} : j\in[1,\lambda]\}$. Hence condition~(ii) holds.\\

\noindent (iii) Property of intersection.
For any $A\in\mathcal{A}$ and $B\in\mathcal{B}$, we have $|A\cap B|=1$. Since $\tau_{a,b}$ is a bijection,
\[
|\tau_{c,d}(A)\cap \tau_{c,d}(B)| = |\tau_{c,d}(A\cap B)| = |A\cap B| = 1,
\]
Therefore all three conditions are satisfied and hence
\(
(X, \tau_{c,d}(\mathcal{A}), \tau_{c,d}(\mathcal{B}))
\)
is again a modular balanced partition design.
\end{proof}
\begin{example}
 For the $(9,3,3)$-\MBPD ~presented in Example~$\ref{ex}$, we apply the affine automorphisms of $\mathbb{Z}_9$, namely the maps
\(
\tau_{c,d}(x)=cx+d \pmod 9,
\)
where $c\in\mathbb{Z}_9^{\times}$ and $d\in\mathbb{Z}_9$. The images of the block partitions $\mathcal{A}$ and $\mathcal{B}$ under these automorphisms yield further $(9,3,3)$-\MBPD s. Table~$\ref{tab:cfg}$ lists all modular balanced partition designs obtained in this manner. In particular, each row of Table~$\ref{tab:cfg}$ corresponds to a distinct affine automorphism of $\mathbb{Z}_9$ and the resulting $(9,3,3)$-\MBPD.
\end{example}

\begin{center}

\begin{longtable}{|c|c|c|c|}
\hline
$c\in \mathbb{Z}_{9}^{\times}$ & $cx+d, \, d\in \mathbb{Z}_{9}$ & $\mathcal{A}=\{A_1, A_2, A_3\}$ & $\mathcal{B}=\{B_1, B_2, B_3\}$  \\ 
\hline
\endfirsthead
\hline
\endhead
\hline
\noalign{\smallskip}
\endfoot
\endlastfoot

    \hline 
        \multirow{9}{*}{1} 
        & $x+0$ & $\{\{0,3,6\}, \{1,4,7\}, \{2,5,8\}\}$ & $\{\{0,1,8\}, \{3,4,5\}, \{6,7,2\}\}$ \\ \cline{2-4}
        & $x+1$ & $\{\{1,4,7\}, \{2,5,8\}, \{3,6,0\}\}$ & $\{\{1,2,0\}, \{4,5,6\}, \{7,8,3\}\}$ \\ \cline{2-4}
        & $x+2$ & $\{\{2,5,8\}, \{3,6,0\}, \{4,7,1\}\}$ & $\{\{2,3,1\}, \{5,6,7\}, \{8,0,4\}\}$ \\ \cline{2-4}
        & $x+3$ & $\{\{3,6,0\}, \{4,7,1\}, \{5,8,2\}\}$ & $\{\{3,4,2\}, \{6,7,8\}, \{0,1,5\}\}$ \\ \cline{2-4}
        & $x+4$ & $\{\{4,7,1\}, \{5,8,2\}, \{6,0,3\}\}$ & $\{\{4,5,3\}, \{7,8,0\}, \{1,2,6\}\}$ \\ \cline{2-4}
        & $x+5$ & $\{\{5,8,2\}, \{6,0,3\}, \{7,1,4\}\}$ & $\{\{5,6,4\}, \{8,0,1\}, \{2,3,7\}\}$ \\ \cline{2-4}
        & $x+6$ & $\{\{6,0,3\}, \{7,1,4\}, \{8,2,5\}\}$ & $\{\{6,7,5\}, \{0,1,2\}, \{3,4,8\}\}$ \\ \cline{2-4}
        & $x+7$ & $\{\{7,1,4\}, \{8,2,5\}, \{0,3,6\}\}$ & $\{\{7,8,6\}, \{1,2,3\}, \{4,5,0\}\}$ \\ \cline{2-4}
        & $x+8$ & $\{\{8,2,5\}, \{0,3,6\}, \{1,4,7\}\}$ & $\{\{8,0,7\}, \{2,3,4\}, \{5,6,1\}\}$ \\ \hline
        \hline
           \multirow{9}{*}{2} 
        & $2x+0$ & $\{\{0,6,3\}, \{2,8,5\}, \{4,1,7\}\}$ & $\{\{0,2,7\}, \{6,8,1\}, \{3,5,4\}\}$ \\ \cline{2-4} 
         & $2x+1$ & $\{\{1,7,4\}, \{3,0,6\}, \{5,2,8\}\}$ & $\{\{1,3,8\}, \{7,0,2\}, \{4,6,5\}\}$ \\ \cline{2-4}
         & $2x+2$ & $\{\{2,8,5\}, \{4,1,7\}, \{6,3,0\}\}$ & $\{\{2,4,0\}, \{8,1,3\}, \{5,7,6\}\}$ \\ \cline{2-4}
          & $2x+3$ & $\{\{3,0,6\}, \{5,2,8\}, \{7,4,1\}\}$ & $\{\{3,5,1\}, \{0,2,4\}, \{6,8,7\}\}$ \\ \cline{2-4}
          & $2x+4$ & $\{\{4,1,7\}, \{6,3,0\}, \{8,5,2\}\}$ & $\{\{4,6,2\}, \{1,3,5\}, \{7,0,8\}\}$ \\ \cline{2-4}
           & $2x+5$ & $\{\{5,2,8\}, \{7,4,1\}, \{0,6,3\}\}$ & $\{\{5,7,3\}, \{2,4,6\}, \{8,1,0\}\}$ \\ \cline{2-4}
            & $2x+6$ & $\{\{6,3,0\}, \{8,5,2\}, \{1,7,4\}\}$ & $\{\{6,8,4\}, \{3,5,7\}, \{0,2,1\}\}$ \\ \cline{2-4}
             & $2x+7$ & $\{\{7,4,1\}, \{0,6,3\}, \{2,8,5\}\}$ & $\{\{7,0,5\}, \{4,6,8\}, \{1,3,2\}\}$ \\ \cline{2-4}
             & $2x+8$ & $\{\{8,5,2\}, \{1,7,4\}, \{3,0,6\}\}$ & $\{\{8,1,6\}, \{5,7,0\}, \{2,4,3\}\}$ \\ \cline{2-4}
        \hline \hline
         \multirow{9}{*}{4} 
        & $4x+0$ & $\{\{0,3,6\}, \{4,7,1\}, \{8,2,5\}\}$ & $\{\{0,4,5\}, \{3,7,2\}, \{6,1,8\}\}$ \\ \cline{2-4}
        & $4x+1$ & $\{\{1,4,7\}, \{5,8,2\}, \{0,3,6\}\}$ & $\{\{1,5,6\}, \{4,8,3\}, \{7,2,0\}\}$ \\ \cline{2-4}
        & $4x+2$ & $\{\{2,5,8\}, \{6,0,3\}, \{1,4,7\}\}$ & $\{\{2,6,7\}, \{5,0,4\}, \{8,3,1\}\}$ \\ \cline{2-4}
        & $4x+3$ & $\{\{3,6,0\}, \{7,1,4\}, \{2,5,8\}\}$ & $\{\{3,7,8\}, \{6,1,5\}, \{0,4,2\}\}$ \\ \cline{2-4}
        & $4x+4$ & $\{\{4,7,1\}, \{8,2,5\}, \{3,6,0\}\}$ & $\{\{4,8,0\}, \{7,2,6\}, \{1,5,3\}\}$ \\ \cline{2-4}
        & $4x+5$ & $\{\{5,8,2\}, \{0,3,6\}, \{4,7,1\}\}$ & $\{\{5,0,1\}, \{8,3,7\}, \{2,6,4\}\}$ \\ \cline{2-4}
         & $4x+6$ & $\{\{6,0,3\}, \{1,4,7\}, \{5,8,2\}\}$ & $\{\{6,1,2\}, \{0,4,8\}, \{3,7,5\}\}$ \\ \cline{2-4}
        & $4x+7$ & $\{\{7,1,4\}, \{2,5,8\}, \{6,0,3\}\}$ & $\{\{7,2,3\}, \{1,5,0\}, \{4,8,6\}\}$ \\ \cline{2-4}
         & $4x+8$ & $\{\{8,2,5\}, \{3,6,0\}, \{7,1,4\}\}$ & $\{\{8,3,4\}, \{2,6,1\}, \{5,0,7\}\}$ \\ \cline{2-4}
        \hline \hline
         \multirow{9}{*}{5} 
        & $5x+0$ & $\{\{0,6,3\}, \{5,2,8\}, \{1,7,4\}\}$ & $\{\{0,5,4\}, \{6,2,7\}, \{3,8,1\}\}$ \\ \cline{2-4}
        & $5x+1$  & $\{\{1,7,4\}, \{6,3,0\}, \{2,8,5\}\}$ & $\{\{1,6,5\}, \{7,3,8\}, \{4,0,2\}\}$ \\ \cline{2-4}
        & $5x+2$  & $\{\{2,8,5\}, \{7,4,1\}, \{3,0,6\}\}$ & $\{\{2,7,6\}, \{8,4,0\}, \{5,1,3\}\}$ \\ \cline{2-4}
        & $5x+3$  & $\{\{3,0,6\}, \{8,5,2\}, \{4,1,7\}\}$ & $\{\{3,8,7\}, \{0,5,1\}, \{6,2,4\}\}$ \\ \cline{2-4}
        & $5x+4$  & $\{\{4,1,7\}, \{0,6,3\}, \{5,2,8\}\}$ & $\{\{4,0,8\}, \{1,6,2\}, \{7,3,5\}\}$ \\ \cline{2-4}
        & $5x+5$  & $\{\{5,2,8\}, \{1,7,4\}, \{6,3,0\}\}$ & $\{\{5,1,0\}, \{2,7,3\}, \{8,4,6\}\}$ \\ \cline{2-4}
        & $5x+6$  & $\{\{6,3,0\}, \{2,8,5\}, \{7,4,1\}\}$ & $\{\{6,2,1\}, \{3,8,4\}, \{0,5,7\}\}$ \\ \cline{2-4}
         & $5x+7$  & $\{\{7,4,1\}, \{3,0,6\}, \{8,5,2\}\}$ & $\{\{7,3,2\}, \{4,0,5\}, \{1,6,8\}\}$ \\ \cline{2-4}
       & $5x+8$  & $\{\{8,5,2\}, \{4,1,7\}, \{0,6,3\}\}$ & $\{\{8,4,3\}, \{5,1,6\}, \{2,7,0\}\}$ \\ \cline{2-4}
        \hline \hline \multirow{9}{*}{7} 
        & $7x+0$ & $\{\{0,3,6\}, \{7,1,4\}, \{5,8,2\}\}$ & $\{\{0,7,2\}, \{3,1,8\}, \{6,4,5\}\}$ \\ \cline{2-4}
        & $7x+1$ & $\{\{1,4,7\}, \{8,2,5\}, \{6,0,3\}\}$ & $\{\{1,8,3\}, \{4,2,0\}, \{7,5,6\}\}$ \\ \cline{2-4}
         & $7x+2$ & $\{\{2,5,8\}, \{0,3,6\}, \{7,1,4\}\}$ & $\{\{2,0,4\}, \{5,3,1\}, \{8,6,7\}\}$ \\ \cline{2-4}
        & $7x+3$ & $\{\{3,6,0\}, \{1,4,7\}, \{8,2,5\}\}$ & $\{\{3,1,5\}, \{6,4,2\}, \{0,7,8\}\}$ \\ \cline{2-4}
     & $7x+4$ & $\{\{4,7,1\}, \{2,5,8\}, \{0,3,6\}\}$ & $\{\{4,2,6\}, \{7,5,3\}, \{1,8,0\}\}$ \\ \cline{2-4}
       & $7x+5$ & $\{\{5,8,2\}, \{3,6,0\}, \{1,4,7\}\}$ & $\{\{5,3,7\}, \{8,6,4\}, \{2,0,1\}\}$ \\ \cline{2-4}
       & $7x+6$ & $\{\{6,0,3\}, \{4,7,1\}, \{2,5,8\}\}$ & $\{\{6,4,8\}, \{0,7,5\}, \{3,1,2\}\}$ \\ \cline{2-4}
      & $7x+7$ & $\{\{7,1,4\}, \{5,8,2\}, \{3,6,0\}\}$ & $\{\{7,5,0\}, \{1,8,6\}, \{4,2,3\}\}$ \\ \cline{2-4}
       & $7x+8$ & $\{\{8,2,5\}, \{6,0,3\}, \{4,7,1\}\}$ & $\{\{8,6,1\}, \{2,0,7\}, \{5,3,4\}\}$\\ \hline
        \hline
         \multirow{9}{*}{8} 
        & $8x+0$ & $\{\{0,6,3\}, \{8,5,2\}, \{7,4,1\}\}$ & $\{\{0,8,1\}, \{6,5,4\}, \{3,2,7\}\}$ \\ \cline{2-4}
        & $8x+1$ & $\{\{1,7,4\}, \{0,6,3\}, \{8,5,2\}\}$ & $\{\{1,0,2\}, \{7,6,5\}, \{4,3,8\}\}$ \\ \cline{2-4}
        & $8x+2$ & $\{\{2,8,5\}, \{1,7,4\}, \{0,6,3\}\}$ & $\{\{2,1,3\}, \{8,7,6\}, \{5,4,0\}\}$ \\ \cline{2-4}
        & $8x+3$ & $\{\{3,0,6\}, \{2,8,5\}, \{1,7,4\}\}$ & $\{\{3,2,4\}, \{0,8,7\}, \{6,5,1\}\}$ \\ \cline{2-4}
        &  $8x+4$ & $\{\{4,1,7\}, \{3,0,6\}, \{2,8,5\}\}$ & $\{\{4,3,5\} ,\{1,0,8\}, \{7,6,2\}\}$ \\ \cline{2-4}
        & $8x+5$ & $\{\{5,2,8\}, \{4,1,7\}, \{3,0,6\}\}$ & $\{\{5,4,6\} ,\{2,1,0\}, \{8,7,3\}\}$ \\ \cline{2-4}
        & $8x+6$ & $\{\{6,3,0\}, \{5,2,8\}, \{4,1,7\}\}$ & $\{\{6,5,7\} ,\{3,2,1\}, \{0,8,4\}\}$ \\ \cline{2-4}
        & $8x+7$ & $\{\{7,4,1\}, \{6,3,0\}, \{5,2,8\}\}$ & $\{\{7,6,8\} ,\{4,3,2\}, \{1,0,5\}\}$ \\ \cline{2-4}
        & $8x+8$ & $\{\{8,5,2\}, \{7,4,1\}, \{6,3,0\}\}$ & $\{\{8,7,0\} ,\{5,4,3\}, \{2,1,6\}\}$ \\ \hline
\caption{A collection of $(9,3,3)$-MBPDs obtained from the affine automorphisms of $\mathbb{Z}_9$. For each affine automorphism $\tau_{c,d}(x)=cx+d \pmod 9$, where $c\in\mathbb{Z}_9^{\times}$ and $d\in\mathbb{Z}_9$, the corresponding row gives a distinct $(9,3,3)$-MBPD. Here, $\mathcal{A}=\{A_1,A_2,A_3\}$ and $\mathcal{B}=\{B_1,B_2,B_3\}$ are partitions of $\mathbb{Z}_9$ into three blocks of size $3$.} 
\label{tab:cfg}
\end{longtable}
\end{center}

\begin{proposition}\label{prop:MBPDauto}
Let $\kappa$ and $\lambda$ be two odd integers, and let
$(X,\mathcal A,\mathcal B)$ be a
$(\nu,\kappa,\lambda)$-modular balanced partition design over
$\mathbb Z_{\nu}$, where $\nu=\kappa\lambda$, obtained from the subgroup
magic rectangle in \textnormal{Theorem~6~\cite{karthik2025existence}}. Let
\(
\mathcal A=\{A_i:i\in[1,\kappa]\}.
\) and \(
\mathcal B=\{B_j:j\in[1,\lambda]\}.
\)
Fix a block $A_i\in\mathcal A$. For each $\ell\in[1,\lambda]$, define
\(
\psi_\ell(\mathcal B)
=
\{\psi_\ell(B):B\in\mathcal B\},
\)
where
\(
\psi_\ell(B)
=
(B\cap A_i)\cup
\{(x+\ell\lambda)\bmod{\nu}:x\in B\setminus A_i\}.
\)
Then
\(
(X,\mathcal A,\psi_\ell(\mathcal B))
\)
is also a $(\nu,\kappa,\lambda)$-modular balanced partition design.
\end{proposition}
\begin{proof}
By Theorem~6~\cite{karthik2025existence}, the blocks of $\mathcal A$ and
$\mathcal B$ are obtained from a subgroup magic rectangle over
$\mathbb Z_{\nu}$. Hence each block of $\mathcal B$ contains exactly one
element from every block of $\mathcal A$.
The map $\psi_\ell$ fixes the unique element in $A_i$ and translates every
other element by $\ell\lambda$ modulo $\nu$. By the construction in
Theorem~6, this translation preserves the partition in $\mathcal A$.
Consequently, the intersection of every block of $\psi_\ell(\mathcal B)$ with each block of $\mathcal A$ contains exactly one element. Since translation modulo $\nu$ is a bijection, $\psi_\ell(\mathcal B)$ is a
partition of $X$ into $\lambda$ blocks of size $\kappa$. Moreover, every
block sum is changed by the same constant modulo $\nu$, so the modular
balance condition is preserved. Therefore,
$(X,\mathcal A,\psi_\ell(\mathcal B))$ is again a
$(\nu,\kappa,\lambda)$-modular balanced partition design.
\end{proof}
\subsection{Affine Group Actions and Classification of $\MBPD$s on $\mathbb Z_\nu$}

In this section, a natural action of the affine automorphism group $\AG(\nu)$ on the set of all modular balanced partition designs over $\mathbb Z_\nu$, is described. Note that when one classifies designs, this action is a fundamental property to  understand an isomorphism, an orbit structure, and the reduction of the search space.
\begin{lemma}
\label{lem:AG-action}
The affine group $\AG(\nu)$ acts on the set of all modular balanced partition designs with point-set $\mathbb Z_\nu$ via
\(
\tau_{c,d}(\mathbb Z_\nu,\mathcal A,\mathcal B)
=
\bigl(\mathbb Z_n,\tau_{c,d}(\mathcal A),\tau_{c,d}(\mathcal B)\bigr),
\)
where
\(
\tau_{c,d}(\mathcal A)
=
\{\tau_{c,d}(A):A\in\mathcal A\}
~\text{and}~
\tau_{c,d}(\mathcal B)
=
\{\tau_{c,d}(B):B\in\mathcal B\}.
\)
For any modular balanced partition design
\(
\mathbb M=(\mathbb Z_\nu,\mathcal A,\mathcal B),
\)
the stabilizer of $\mathbb M$,
\(
\operatorname{Stab}_{\AG(\nu)}(\mathbb M)
=
\{\tau\in\AG(\nu):\tau(\mathbb M)=\mathbb M\}
\) 
is a subgroup of $\AG(\nu)$ under this action.
\end{lemma}
\begin{proof}
Let
\(
\mathbb M=(\mathbb Z_\nu,\mathcal A,\mathcal B)
\) be a modular balanced partition design.
By Theorem~\ref{thm5.10}, for every
$\tau_{c,d}\in \AG(\nu)$,
the image
\(
\tau_{c,d}(\mathbb M)
\)
is again a modular balanced partition design. Hence, the action is well-defined.
Suppose if $\tau_{c,d},\tau_{e,f}\in\AG(\nu)$, then 
\[
\tau_{c,d}\bigl(\tau_{e,f}(\mathbb M)\bigr)
=
(\tau_{c,d}\circ\tau_{e,f})(\mathbb M),
\]
so the compatibility condition for a group action holds. Moreover, the identity element
$\tau_{1,0}\in\AG(\nu)$
satisfies
\(
\tau_{1,0}(\mathbb M)=\mathbb M.
\)
Therefore, $\AG(\nu)$ acts on the set of all modular balanced partition designs with point-set $\mathbb Z_\nu$. Now, the stabilizer of $\mathbb M$ is
\[
\operatorname{Stab}_{\AG(\nu)}(\mathbb M)
=
\{\tau\in\AG(\nu):\tau(\mathbb M)=\mathbb M\}.
\]
Since the stabilizer of any element under a group action is a subgroup of the acting group, it follows that
\(
\operatorname{Stab}_{\AG(\nu)}(\mathbb M)
\leq \AG(\nu).
\)
\end{proof}

\begin{theorem}
\label{thm:iso-count}
Let $\mathcal S$ denote the finite set of all $\MBPD$s over $\mathbb Z_\nu$. Then the number of non-isomorphic $\MBPD$s over $\mathbb Z_\nu$ is equal to the number of orbits of the natural action of $\AG(\nu)$ on $\mathcal S$. In addition,
\[
\#\{\text{non-isomorphic $\MBPD$s on }\mathbb Z_\nu\}
=
\frac{1}{|\AG(\nu)|}
\sum_{\tau\in \AG(\nu)}
|\operatorname{Fix}(\tau)|,
\]
where
\(
\operatorname{Fix}(\tau)
=
\{\mathbb M\in\mathcal S:\tau(\mathbb M)=\mathbb M\}
\) is the set of designs fixed by the affine automorphism $\tau$ In particular, if every $\MBPD$ in $\mathcal S$ has trivial stabilizer in $\AG(\nu)$, then
\[
\#\{\text{non-isomorphic $\MBPD$s on }\mathbb Z_\nu\}
=
\frac{|\mathcal S|}{|\AG(\nu)|}.
\]
\end{theorem}
\begin{proof}
Two $\MBPD$s are isomorphic if and only if they belong to the same orbit under the action of $\AG(\nu)$ on $\mathcal S$. Therefore, the set of all isomorphic classes of $\MBPD$s is precisely the orbit space $\mathcal S/\AG(\nu)$.
Since both $\mathcal S$ and $\AG(\nu)$ are finite, by Burnside's Lemma \ref{lem:burnside}, the number of orbits is
$
|\mathcal S/\AG(\nu)|
=
\frac{1}{|\AG(\nu)|}
\sum_{\tau\in \AG(\nu)}
|\operatorname{Fix}(\tau)|,
$
where
\(
\operatorname{Fix}(\tau)
=
\{\mathbb M\in\mathcal S:\tau(\mathbb M)=\mathbb M\}.
\)
This, proves the stated formula. If every $\MBPD$ has trivial stabilizer in $\AG(\nu)$, then by Orbit--Stabilizer Theorem \ref{ost} each orbit has cardinality
$
\frac{|\AG(\nu)|}
{|\operatorname{Stab}_{\AG(\nu)}(\mathbb M)|}
=
|\AG(\nu)|.$
Since the set of all orbits partition $\mathcal S$, the number of non-isomorphic $\MBPD$s on $\mathbb{Z}_\nu$ is 
\[
|\mathcal S/\AG(\nu)|
=
\frac{|\mathcal S|}{|\AG(\nu)|}.
\]
 
\end{proof}

\begin{lemma}
\label{thm:orbit-count}
If $\mathbb M=(\mathbb Z_\nu,\mathcal A,\mathcal B)$ is an $\MBPD$, then the orbit of $\mathbb M$ under $\AG(\nu)$ has size
\[
|\AG(\nu)\mathbb M|
= \frac{|\AG(\nu)|}{\big|\operatorname{Stab}_{\AG(\nu)}(\mathbb M)\big|}.
\]
\end{lemma}

\begin{proof}
The result follows from  Orbit--Stabilizer Theorem \ref{ost} and Lemma~\ref{lem:AG-action}.
\end{proof}

\paragraph{Search Reduction.}
When performing computational or manual classification, it suffices to inspect one representative from each $AG(\nu)$-orbit.  
Designs in the same orbit are isomorphic via affine maps, so selecting one per orbit yields a complete set of isomorphism types while reducing the search space by the factor given in Lemma~\ref{thm:orbit-count}.

\begin{proposition}
\label{prop:translation-regular}
Let $T=\{\tau_{1,b}:b\in\mathbb Z_\nu\}$ be the subgroup of translations.  
If a $\MBPD$ $\mathbb M=(\mathbb Z_\nu,\mathcal A,\mathcal B)$ is invariant under $T$, then $\mathbb M$ is a cyclic design.  
That is, there exist sets of base blocks $\mathcal R\subseteq\mathcal A$ and $\mathcal S\subseteq\mathcal B$ such that
\(
\mathcal A=\{\tau_{1,b}(R): R\in\mathcal R, \; b\in\mathbb Z_\nu\}, \textnormal{~and~}
\mathcal B=\{\tau_{1,b}(S): S\in\mathcal S, \; b\in\mathbb Z_\nu\}.
\)
\end{proposition}
\begin{proof}
Since $\mathbb M$ is invariant under the action of $T$, the collections $\mathcal A$ and $\mathcal B$ are unions of $T$-orbits. Let $\mathcal R\subseteq\mathcal A$ be a set containing exactly one representative from each $T$-orbit on $\mathcal A$, and let $\mathcal S\subseteq\mathcal B$ be a set containing exactly one representative from each $T$-orbit on $\mathcal B$.
Then every block of $\mathcal A$ belongs to the orbit of a unique representative $R\in\mathcal R$. Hence
\(
\mathcal A
=
\{\tau_{1,b}(R): R\in\mathcal R,\ b\in\mathbb Z_\nu\}.
\)
Similarly,
\(
\mathcal B
=
\{\tau_{1,b}(S): S\in\mathcal S,\ b\in\mathbb Z_\nu\}.
\)
Therefore, all blocks of $\mathcal A$ and $\mathcal B$ are obtained by developing the base blocks in $\mathcal R$ and $\mathcal S$ under the action of the translation group $T$. Consequently, $\mathbb M$ is a cyclic design.
\end{proof}

\begin{proposition}
\label{prop:mult-invariance}
Let $H\le\mathbb Z_\nu^\times$ be a multiplicative subgroup, and let  
$H_0=\{\tau_{a,0}: a\in H\}\le AG(\nu)$.  
If $H_0\subseteq\operatorname{Stab}_{AG(\nu)}(\mathcal D)$ for a $\MBPD$ $\mathcal D=(\mathbb Z_\nu,\mathcal A,\mathcal B)$, then every block of $\mathcal A$ and $\mathcal B$ is a union of $H$-orbits in $\mathbb Z_\nu$.  
Consequently, block sizes are sums of orbit sizes, and each block is invariant under the subgroup that stabilizes it.
\end{proposition}

\begin{proof}
If $\tau_{a,0}$ fixes the design for each $a\in H$, then for any block $A\in\mathcal A$,
\(
\tau_{a,0}(A)=A.
\)
Thus $A$ is a union of $H$-orbits.  
The conclusion for block sizes follows immediately.
\end{proof}

% \begin{definition}
% Two $\MBPD$s $\mathbb{M}_1$ and $\mathbb{M}_2$ on $\mathbb Z_\nu$ are \emph{affine-equivalent} if there exists $\tau\in AG(\nu)$ such that 
% \[
% \tau\cdot  \mathbb{M}_1 = \mathbb{M}_2.
% \]
% \end{definition}

 \section{Modular balanced partition design: A tool to construct transversal design} 
In this section, we assume $X=\mathbb{Z}_{\nu}$, where $\nu=\kappa\lambda$ and
$\kappa$ and $\lambda$ are odd integers. Using a modular balanced partition
design, we construct a transversal design $TD(\kappa,\lambda)$ as follows.

\medskip

\noindent\textbf{Step 1.}
By Theorem~\ref{thmfre}, let
$(X,\mathcal A,\mathcal B)$ be a
$(\nu,\kappa,\lambda)$-modular balanced partition design, where
\(
\mathcal A=\{A_i:i\in[1,\kappa]\}
~\text{and}~
\mathcal B=\{B_j:j\in[1,\lambda]\}.
\)

\medskip

\noindent\textbf{Step 2.}
Fix the block $A_1\in\mathcal A$. For each integer
$\ell\in[1,\lambda]$, let
\(
(X,\mathcal A,\mathcal B_\ell)
\)
be the $(\nu,\kappa,\lambda)$-modular balanced partition design obtained from
Proposition~\ref{prop:MBPDauto}. Thus, we obtain $\lambda$ modular balanced
partition designs with block collections
\(
\mathcal B_1,\mathcal B_2,\ldots,\mathcal B_\lambda.
\)
\noindent Let
\[
\mathscr B=\bigcup_{\ell=1}^{\lambda}\mathcal B_\ell.
\]
\noindent Then,
\(
|\mathscr B|=\lambda^2.
\) Therefore $(X,\mathcal{A},\mathscr {B})$ satisfies all the conditions of the transversal design $TD(k,n)$.
   
% \end{enumerate}
%   The construction is given by illustration. 
% From table \ref{tab:cfg}, fix $a_0=2$ and $b_0=3$.\\
% Then, \[\mathcal{A}=\{\{3,0,6\}, \{5,2,8\}, \{7, 4,1\}\},\] 
% and, \begin{align*}\mathcal{C}&=\{\{2x+0\}, \{2x+1\},\ldots, \{2x+8\}\}\\
% & =\{\{0,2,7\}, \{6,8,1\}, \{3,5,4\}, \{1,3,8\}, \{7,0,2\}, \{4,6,5\}, \{2,4,0\}, \{8,1,3\}, \{5,7,6\}, \{3,5,1\},\\& ~~~~~~~~ \{0,2,4\}, \{6,8,7\}, \{4,6,2\}, \{1,3,5\}, \{7,0,8\}, \{5,7,3\},  \{2,4,6\}, \{8,1,0\}, \{6,8,4\}, \{3,5,7\}, \\& ~~~~~~~~\{0,2,1\}, \{7,0,5\}, \{4,6,8\}, \{1,3,2\}, \{8,1,6\}, \{5,7,0\}, \{2,4,3\}\}\end{align*}
% \noindent Then $(X,\mathcal{A},\mathcal{C})$ satisfies all the conditions of the transversal design $TD_5(3,3)$.
 \begin{illustration}
         Consider a $(27,3,9)$ modular balanced partition design $(X, \mathcal{A}, \mathcal{B})$. Suppose that $\mathcal{A}=\{A_1, A_2, A_3\}$ in which $A_1,A_2, A_3$ are defined below.
          \begin{align*}
             A_1& =\{0, 3,6,9,12,15,18, 21,24\},\\
             A_2 & =\{1,4,7,10,13,16,19,22, 25\},\\
             A_3 & =\{26,23,20, 17,14,11,8,5,2\}.
         \end{align*}
         Now, if we define the $\lambda$ maps, then blocks in each $\mathcal{B}_i$ are given below, 
\small{\begin{align*}
   \mathcal{B}_1& = \{ \{0,1,26\}, \{3,4,23\}, \{6,7,20\}, \{9,10,17\},\{12,13,14\},\{15,16,11\}, \{18,19,8\},\{21,22,5\}, \{24,25,2\}\}\\
     \mathcal{B}_2& = \{\{0,4,17\}, \{3,7,14\},\{6,10,11\}, \{9,13,8\}, \{12,16,5\}, \{15,19,2\}, \{18,22,26\}, \{21,25,23\}, \{24,1,20\}\}\\
     \mathcal{B}_3& =\{ \{0,7,26\}, \{3,10,23\}, \{6,13,20\}, \{9,16,17\}, \{12,19,14\}, \{15, 22, 11\}, \{18,25,8\}, \{21,1,5\}, \{24,4,2\}\}\\
     \mathcal{B}_4& = \{\{0,10,8\},\{3,13,5\}, \{6,16,2\}, \{9, 19, 26\}, \{12, 22, 23\}, \{15,25,20\}, \{18,1,17\}, \{21,4,14\}, \{24,7,11\}\}\\
     \mathcal{B}_5 &=\{\{0,13,17\}, \{3,16,14\}, \{6, 19,11\}, \{9, 22,8\}, \{12,25,5\}, \{15, 1,2\}, \{18,4,26\}, \{21,7,23\}, \{24,10,20\}\}\\
      \mathcal{B}_6 &=\{\{0,16,26\}, \{3,19,23\}, \{6, 22,20\}, \{9, 25,17\}, \{12,1,14\}, \{15,4,11\}, \{18,7,8\}, \{21,10,5\}, \{24,13,2\}\}\\
      \mathcal{B}_7 &=\{\{0,19,8\}, \{3,22,5\}, \{6, 25,2\}, \{9, 1,26\}, \{12,4,23\}, \{15,7,20\}, \{18,10,17\}, \{21,13,14\}, \{24,16,11\}\}\\
       \mathcal{B}_8 &=\{\{0,22,17\}, \{3,25,14\}, \{6, 1,11\}, \{9, 4,8\}, \{12,7,5\}, \{15,10,2\}, \{18,13,26\}, \{21,16,23\}, \{24,19,20\}\}\\
        \mathcal{B}_9 &=\{\{0,25,26\}, \{3,1,23\}, \{6, 4,20\}, \{9, 7,17\}, \{12,10,14\}, \{15,13,11\}, \{18,16,8\}, \{21,19,5\}, \{24,22,2\}\}
      \end{align*}}
\noindent If $\displaystyle \mathscr{B}=\bigcup_{i=1}^9 \mathcal{B}_i$, then $(X,\mathcal{A}, \mathscr{B})$ is a transversal design of size $9$, block group size $3$, and index $1$. 
 \end{illustration}
  
\section{Conclusion and scope}
In this article, we introduce the notion of {modular balanced partition design} and examine several of their fundamental properties together with the necessary conditions for their existence. Many of these conditions, however, remain unresolved and present open problems for further investigation. 
We prove the existence of $(mn, m, n)$-MBPDs over $\mathbb{Z}_{mn}$ for the case where both $m$ and $n$ are odd. We emphasize a construction of $(n^{2}, n, n)$-MBPDs over $\mathbb{Z}_{n^{2}}$ when $n \equiv 0 \pmod{4}$. In contrast, when $n \equiv 2 \pmod{4}$, the construction exhibits an inherent non-symmetry in most of the $n \times n$ subarrays. Consequently, for even $m$ and $n$ with $m \neq n$, the existence of $(mn, m, n)$-MBPDs over $\mathbb{Z}_{mn}$ remains an open question. Furthermore, the concepts of automorphisms and affine automorphisms within the framework of modular balanced partition designs are introduced and studied. Finally, a method to  construct a transversal design by using a modular balanced partition design, is obtained.
%%%%%%%%%%%%%%%%

\section*{Conflicts of Interest}
The authors declare that there is no conflict of interest.

\newpage

\begin{thebibliography}{99}

\bibitem{cameron1999permutation}
P. J. Cameron,
\emph{Permutation Groups},
London Mathematical Society Student Texts, Vol. 45,
Cambridge University Press, Cambridge, 1996.


\bibitem{zbMATH05233000}
P. J. Cameron,
\emph{Introduction to Algebra},
2nd ed.,
Oxford University Press, New York, 2008.

\bibitem{zbMATH00050655}
P. J. Cameron and J. H. van Lint,
\emph{Designs, Graphs, Codes and Their Links},
London Mathematical Society Student Texts, Vol. 22,
Cambridge University Press, Cambridge, 1991.


\bibitem{MR3720542}
S. Cichacz,
On zero sum-partition of Abelian groups into three sets and group distance magic labeling,
\emph{Ars Math. Contemp.} \textbf{13}(2) (2017), 417--425.


\bibitem{MR3761934}
S. Cichacz,
Zero sum partition of Abelian groups into sets of the same order and its applications,
\emph{Electron. J. Combin.} \textbf{25}(1) (2018), \#P1.20.


\bibitem{MR4149158}
S. Cichacz and T. F. Dinc,
A magic rectangle set on Abelian groups and its application,
\emph{Discrete Appl. Math.} \textbf{288} (2021), 198--210.


\bibitem{MR1370815}
C. J. Colbourn,
Transversal designs of block size eight and nine,
\emph{European J. Combin.} \textbf{17}(1) (1996), 1--14.


\bibitem{MR1435524}
C. J. Colbourn,
Some direct constructions for incomplete transversal designs,
\emph{J. Statist. Plann. Inference} \textbf{56}(1) (1996), 93--104.


\bibitem{MR2239315}
D. Combe and A. M. Nelson,
Magic labellings of infinite graphs over infinite groups,
\emph{Australas. J. Combin.} \textbf{35} (2006), 193--210.

\bibitem{MR2569784}
J. P. De Los Reyes, A. Das, C. K. Midha, and P. Vellaisamy,
On a method to construct magic rectangles of even order,
\emph{Util. Math.} \textbf{80} (2009), 277--284.


\bibitem{Erdos1961}
P. Erd\H{o}s, A. Ginzburg, and A. Ziv,
A theorem in additive number theory,
\emph{Bull. Res. Council Israel} \textbf{10F} (1961), 41--43.

\bibitem{MR4145425}
B. Freyberg,
On constant sum partitions and applications to distance magic-type graphs,
\emph{AKCE Int. J. Graphs Comb.} \textbf{17}(1) (2020), 373--379.


\bibitem{MR2313123}
W. Gao and A. Geroldinger,
Zero-sum problems in finite Abelian groups: A survey,
\emph{Expo. Math.} \textbf{24}(4) (2006), 337--369.


\bibitem{zbMATH05018184}
A. Geroldinger and F. Halter-Koch,
\emph{Non-Unique Factorizations: Algebraic, Combinatorial and Analytic Theory},
Pure and Applied Mathematics, Vol. 278,
Chapman \& Hall/CRC, Boca Raton, FL, 2006.

\bibitem{zbMATH06162097}
D. J. Grynkiewicz,
\emph{Structural Additive Theory},
Developments in Mathematics, Vol. 30,
Springer, Cham, 2013.


\bibitem{MR2510327}
G. Kaplan, A. Lev, and Y. Roditty,
On zero-sum partitions and anti-magic trees,
\emph{Discrete Math.} \textbf{309}(8) (2009), 2010--2014.


\bibitem{karthik2025existence}
S. Karthik, A. Venkatesan, and K. Paramasivam,
On the existence of a subgroup magic rectangle,
In: \emph{International Workshop on Combinatorial Algorithms (IWOCA 2025)},
Lecture Notes in Computer Science, Springer, (2025), pp. 332--346.


\bibitem{MR2469212}
C. C. Lindner and C. A. Rodger,
\emph{Design Theory},
Discrete Mathematics and Its Applications,
2nd ed., CRC Press, Boca Raton, FL, 2009.

\bibitem{MR4423367}
J. H. Silverman,
\emph{Abstract Algebra---An Integrated Approach},
Pure and Applied Undergraduate Texts, Vol. 55,
American Mathematical Society, Providence, RI, 2022.

\bibitem{MR2029249}
D. R. Stinson,
\emph{Combinatorial Designs: Constructions and Analysis},
Springer-Verlag, New York, 2004.


\bibitem{MR1068318}
S. Warner,
\emph{Modern Algebra},
2nd ed.,
Dover Books on Advanced Mathematics,
Dover Publications, New York, 1990.

\bibitem{MR3418508}
X. Zeng,
On zero-sum partitions of Abelian groups,
\emph{Integers} \textbf{15} (2015), \#A44.
\end{thebibliography}
\end{document}